\documentclass[final]{siamltex}

\def\d{\delta} 
 
\def\e{{\epsilon}}

\def\norm#1{\|#1\|} 
\def\wb{{\bar w}}

\usepackage{algorithm}
\usepackage{algorithmic}
\usepackage{amsmath}
\usepackage{amssymb}
\usepackage{caption}
\usepackage{comment}
\usepackage{graphicx}
\usepackage{float}
\usepackage[latin1]{inputenc}
\usepackage[pdfborder={0 0 0}, colorlinks=true, linkcolor=blue, citecolor=red]{hyperref}\hypersetup{pdfborder=0 0 0}
\usepackage{mathrsfs}
\usepackage{multicol}
\usepackage{multirow}
\usepackage{showkeys,cite}
\usepackage{subfigure}
\usepackage{wrapfig}
\usepackage{ulem}
\newcommand{\inred}[1]{{\color{red} #1}}

\usepackage{soul,xargs}
\usepackage[pdftex,dvipsnames]{xcolor}

\usepackage[colorinlistoftodos,prependcaption,textsize=tiny]{todonotes}
\newcommandx{\question}[2][1=]{\todo[linecolor=red,backgroundcolor=red!20,bordercolor=red,#1]{#2}}
\newcommandx{\change}[2][1=]{\todo[linecolor=blue,backgroundcolor=blue!25,bordercolor=blue,#1]{#2}}
\newcommandx{\info}[2][1=]{\todo[linecolor=OliveGreen,backgroundcolor=OliveGreen!25,bordercolor=OliveGreen,#1]{#2}}
\newcommandx{\improve}[2][1=]{\todo[linecolor=Plum,backgroundcolor=Plum!25,bordercolor=Plum,#1]{#2}}
\newcommandx{\answer}[2][1=]{\todo[linecolor=blue,backgroundcolor=White!25,bordercolor=Plum,#1]{#2}}
\newcommandx{\suggest}[2][1=]{\todo[linecolor=blue,backgroundcolor=White!25,bordercolor=red,#1]{#2}}

\graphicspath{{./plots/}}
\begin{document}

\title{Unified geometric multigrid algorithm for hybridized high-order finite element methods}

    \author{Tim Wildey\thanks{ Sandia National Laboratories, Center for Computing Research ({\tt tmwilde@sandia.gov}). The views expressed in the article do not necessarily represent the views of the U.S. Department of Energy or the United States Government.  Sandia National Laboratories is a multimission laboratory managed and operated by National Technology and Engineering Solutions of Sandia, LLC., a wholly owned subsidiary of Honeywell International, Inc., for the U.S. Department of Energy's National Nuclear Security Administration under contract DE-NA-0003525.} \and Sriramkrishnan Muralikrishnan\thanks{Department of Aerospace Engineering and Engineering Mechanics, The University of Texas at Austin, Austin, TX 78712, USA.} \and Tan Bui-Thanh\thanks{Department of Aerospace Engineering and Engineering Mechanics, and the Institute for Computational Engineering and Sciences, The University of Texas at Austin, Austin, TX 78712, USA.} }

\bibliographystyle{siam}
\newcommand{\TODO}[1]{ \fbox{\parbox{3in}{\bf TODO: #1}}}

\newcommand{\grbf}[1] {\mbox{\boldmath${#1}$\unboldmath}}
\newcommand{\gbf}[1] {\mathbf{#1}}

\newcommand{\beq} {\begin{equation}}
\newcommand{\eeq} {\end{equation}}
\newcommand{\bdm} {\begin{displaymath}}
\newcommand{\edm} {\end{displaymath}}
\newcommand{\bit}{\begin{itemize}}
\newcommand{\eit}{\end{itemize}}
\newcommand{\bde}{\begin{description}}
\newcommand{\ede}{\end{description}}
\newcommand{\bce}{\begin{center}}
\newcommand{\ece}{\end{center}}
\newcommand{\ben} {\begin{enumerate}}
\newcommand{\een} {\end{enumerate}}
\newcommand{\bea} {\begin{eqnarray}}
\newcommand{\eea} {\end{eqnarray}}
\newcommand{\barr} {\begin{array}}
\newcommand{\earr} {\end{array}}
\newcommand{\bean} {\begin{eqnarray*}}
\newcommand{\eean} {\end{eqnarray*}}
\newcommand{\edoc} {\end{document}}
\newcommand{\On}{\hspace{0.2cm} \mbox{on} \hspace{0.2cm}}
\newcommand{\In}{\hspace{0.2cm} \mbox{in} \hspace{0.2cm}}
\newcommand{\Minimize}{\hspace{0.2cm} \mbox{minimize} \hspace{0.2cm}}
\newcommand{\Maximize}{\hspace{0.2cm} \mbox{maximize} \hspace{0.2cm}}
\newcommand{\SubjectTo}{\hspace{0.2cm} \mbox{subject} \hspace{0.15cm}
            \mbox{to} \hspace{0.2cm}}

\def\etal{{\it et al.~}}

\newcommand{\point}[1] {\ensuremath{\boldsymbol{#1}}}
\newcommand{\vvec}[1] {\ensuremath{\boldsymbol{#1}}}
\renewcommand{\vec}[1] {\ensuremath{\boldsymbol{#1}}}
\newcommand{\cvec}[1] {\ensuremath{\boldsymbol{{#1}}}}
\newcommand{\ten}[1] {\ensuremath{\boldsymbol{#1}}}
\newcommand{\tenfour}[1] {\ensuremath{\boldsymbol{\mathsf{#1}}}}
\newcommand{\comp}[1]{\begin{pmatrix}{ #1 }\end{pmatrix}}
\newcommand{\vv}{\ensuremath{\vec{v}}}
\newcommand{\vr}{\left(\ensuremath{\vec{r}}\right)}
\newcommand{\att}{\left(t\right)}
\newcommand{\tvv}{\ensuremath{\tilde{\vec{v}}}}
\newcommand{\ww}{\ensuremath{\vec{w}}}
\newcommand{\GG}{\ensuremath{\ten{G}}}
\newcommand{\FF}{\ensuremath{\boldsymbol{\mathfrak{F}}}}
\newcommand{\tEE}{\ensuremath{\tilde{\vec{E}}}}
\newcommand{\nn}{\ensuremath{\vec{n}}}
\renewcommand{\SS}{\ensuremath{\ten{S}}}
\newcommand{\tSS}{\ensuremath{\tilde{\ten{S}}}}
\newcommand\tK{{\bf K}}
\newcommand{\cp}{\ensuremath{{c_p}}}
\newcommand{\cs}{\ensuremath{{c_s}}}
\newcommand{\tcs}{\ensuremath{{\tilde{c}_s}}}
\newcommand{\Vp}{\ensuremath{{V^e_p}}}
\newcommand{\Vs}{\ensuremath{{V^e_s}}}
\newcommand{\Ss}{\ensuremath{{S^e_s}}}
\newcommand{\nxnx}[1] {\ensuremath{\nn\times\left(\nn\times{#1}\right)}}

\newcommand{\dd}[2] {\ensuremath{\frac{\partial {#1}}{\partial {#2}}}}
\newcommand{\DD}[2] {\ensuremath{\frac{d {#1}}{d {#2}}}}
\newcommand{\Grad} {\ensuremath{\nabla}}
\newcommand{\Gradv}[1]{\ensuremath{\nabla_{#1}}}
\newcommand{\Div} {\ensuremath{\nabla\cdot}}
\newcommand{\Divv}[1] {\ensuremath{\nabla_{#1}\cdot}}
\newcommand{\Curl} {\ensuremath{\nabla\times}}
\newcommand{\Cten}{\ensuremath{\tenfour{C}}}

\newcommand{\eval}[2][\right]{\relax
  \ifx#1\right\relax \left.\fi#2#1\rvert}

\providecommand{\abs}[1]{\left\lvert#1\right\rvert}
\providecommand{\norm}[1]{\left\lVert#1\right\rVert}

\newcommand{\Nel} {\ensuremath{{N_T}}}
\newcommand{\Ned} {\ensuremath{{N_\text{ed}}}}
\newcommand{\De} {\ensuremath{{\mathsf{D}^e}}}
\newcommand{\Dep} {\ensuremath{{\mathsf{D}^{e'}}}}
\newcommand{\Dhat} {\ensuremath{\hat{\mathsf{D}}}}
\newcommand{\Dehat} {\ensuremath{{\hat{\mathsf{D}}^e}}}
\newcommand{\half} {\ensuremath{\frac{1}{2}}}
\newcommand{\IN} {\ensuremath{{\mathcal{I}_N}}}
\newcommand{\INe} {\ensuremath{{\mathcal{I}_{N_e}}}}

\newcommand{\mc}[1]{\mathcal{#1}}
\newcommand{\mcC}[1]{\mathcal{#1}}
\newcommand{\mb}[1]{\mathbf{#1}}
\newcommand{\mbb}[1]{\mathbb{#1}}
\newcommand{\mi}[1]{\mathit{#1}}
\newcommand{\mf}[1]{\mathfrak{#1}}
\newcommand{\Oi}[1]{\int_{\mc{D}} #1 \, d\vec{x}}
\newcommand{\TOi}[1]{\int_{0}^T \int_{\mc{D}} #1 \, d\vec{x}dt}
\newcommand{\Ti}[1]{\int_{0}^T #1 \, dt}
\newcommand{\TDei}[1]{\int_{0}^T\int_{\De} #1 \, d\vec{x}dt}
\newcommand{\TDeid}[1]{\int_{0}^T\int_{\De,N} #1 \, d\vec{x}dt}
\newcommand{\Dei}[1]{\int_{\De} #1 \, d\vec{x}}
\newcommand{\DeI}[1]{\int_{D} #1 \, d\vec{x}}
\newcommand{\DeiM}[1]{\int_{\Dhat} #1 \, d\vec{r}}
\newcommand{\DeMd}[1]{\int_{\De,N_e} #1 \, d\vec{x}}
\newcommand{\DeiMd}[1]{\int_{\Dhat,N_e} #1 \, d\vec{r}}
\newcommand{\DeMdN}[1]{\int_{D,N} #1 \, d\vec{x}}
\newcommand{\DeiMdN}[1]{\int_{\Dhat,N} #1 \, d\vec{r}}
\newcommand{\DeiMdNC}[2]{\int_{\Dhat,N_{#2}} #1 \, d\vec{r}}
\newcommand{\pDei}[1]{\int_{\partial \De} #1 \, d\vec{x}}
\newcommand{\pDeiM}[1]{\int_{\partial \Dhat} #1 \, d\vec{r}}
\newcommand{\pDeiMd}[1]{\int_{\partial \Dhat,N_e} #1 \, d\vec{r}}
\newcommand{\pDeiMdN}[1]{\int_{\partial \Dhat,N} #1 \, d\vec{r}}
\newcommand{\pDeiMdNC}[2]{\int_{\partial \Dhat,N_{#2}} #1 \, d\vec{r}}
\newcommand{\pMortar}[2]{\int_{\mc{M}_{#2}} #1 \, d\vec{r}}
\newcommand{\pMortard}[2]{\int_{\mc{M}^{#2},N_{#2}} #1 \, d\vec{r}}
\newcommand{\pMortardd}[3]{\int_{\mc{M}^{#2},N_{#3}} #1 \, d\vec{r}}
\newcommand{\pMortardh}[3]{\int_{\mc{M}_{#2},N_{#3}} #1 \, d\vec{r}}
\newcommand{\pMortardhn}[4]{\int_{{#2}_{#3},N_{#4}} #1 \, d\vec{r}}
\newcommand{\TpDei}[1]{\int_{0}^T\int_{\partial \De} #1 \, d\vec{x}}
\newcommand{\TpDeid}[1]{\int_{0}^T\int_{\partial \De,N} #1 \, d\vec{x}}
\newcommand{\Deid}[1]{\int_{\De,N_e} #1 \, d\vec{x}}
\newcommand{\DeidN}[1]{\int_{\De,N} #1 \, d\vec{x}}
\newcommand{\pDeidN}[1]{\int_{\partial \De,N} #1 \, d\vec{x}}
\newcommand{\pDeid}[1]{\int_{\partial \De,N_e} #1 \, d\vec{x}}
\newcommand{\nor}[1]{\left\| #1 \right\|}
\newcommand{\snor}[1]{\left| #1 \right|}
\newcommand{\LRp}[1]{\left( #1 \right)}
\newcommand{\LRs}[1]{\left[ #1 \right]}
\newcommand{\LRa}[1]{\left< #1 \right>}
\newcommand{\LRc}[1]{\left\{ #1 \right\}}
\newcommand{\pp}[2]{\frac{\partial #1}{\partial #2}}
\newcommand{\ppcp}[1]{\frac{\partial #1}{\partial \vec{c_p}}}
\newcommand{\ppcpj}[1]{\frac{\partial #1}{\partial c_p}}
\newcommand{\ppcpmf}[1]{\frac{\partial #1}{\partial \vec{\mf{c_p}}}}
\newcommand{\ppm}[1]{\frac{\partial #1}{\partial \vec{m}}}
\newcommand{\ppho}[3]{\frac{\partial^{#3} #1}{{\partial #2}^{#3}}}
\newcommand{\ppmi}[1]{\frac{\partial #1}{\partial m_i}}
\newcommand{\ppmj}[1]{\frac{\partial #1}{\partial m_j}}
\newcommand{\ppcpm}[1]{\frac{\partial #1}{\partial \vec{c_p}^-}}
\newcommand{\ppcpp}[1]{\frac{\partial #1}{\partial \vec{c_p}^+}}
\newcommand{\ppcs}[1]{\frac{\partial #1}{\partial \vec{c_s}}}
\newcommand{\ppcsm}[1]{\frac{\partial #1}{\partial \vec{c_s}^-}}
\newcommand{\ppcsp}[1]{\frac{\partial #1}{\partial \vec{c_s}^+}}
\newcommand{\td}[2]{\frac{d #1}{d #2}}
\newcommand{\Rvert}[1]{\left. #1 \right|}
\newcommand{\bs}[1]{\boldsymbol{#1}}
\newcommand{\ipDed}[3]{\LRp{ #1, #2}_{\De,N}^{#3}}

\newcommand{\bigO}{\mathcal{O}}

\newcommand{\iOm}[1]{\int_{\Omega} #1 \, d\Omega}
\newcommand{\iGb}[1]{\int_{\Gamma_{b}} #1 \, ds}
\newcommand{\iGs}[1]{\int_{\Gamma_{s}} #1 \, ds}
\newcommand{\iGsy}[1]{\int_{\Gamma_{s}} #1 \, ds(\mb{y})}
\newcommand{\RE}{\mbox{{I}}\hspace{-0.5ex}\mbox{{R}}}
\newcommand{\iC}[1]{\int_{0}^{2\pi} #1 \, d\theta}
\newcommand{\iGr}[1]{\int_{\Gamma_{\infty}} #1 \, ds}

\newcommand{\jump}[1] {\ensuremath{[\![{#1}]\!]}}
\newcommand{\diff}[1] {\ensuremath{\LRs{#1}}}
\newcommand{\average}[1] {\ensuremath{\LRc{\!\!\LRc{#1}\!\!}}}

\newcounter{tablecell}

\newcommand*{\numbercell}{%
  \stepcounter{tablecell}%
  \thetablecell. 
}

\newtheorem{propo}{Proposition}
\newtheorem{coro}{Corollary}
\newtheorem{theo}{Theorem}
\newtheorem{lem}{Lemma}

\newtheorem{defi}{definition}

\newtheorem{rema}{remark}

\newcommand{\figlab}[1]{\label{fig:#1}}
\newcommand{\eqnlab}[1]{\label{eq:#1}}
\newcommand{\theolab}[1]{\label{theo:#1}}
\newcommand{\corolab}[1]{\label{coro:#1}}
\newcommand{\propolab}[1]{\label{propo:#1}}
\newcommand{\lemlab}[1]{\label{lem:#1}}
\newcommand{\defilab}[1]{\label{defi:#1}}
\newcommand{\remalab}[1]{\label{rema:#1}}
\newcommand{\tablab}[1]{\label{tab:#1}}

\newcommand{\figref}[1]{\ref{fig:#1}}
\newcommand{\theoref}[1]{\ref{theo:#1}}
\newcommand{\defiref}[1]{\ref{defi:#1}}
\newcommand{\remaref}[1]{\ref{rema:#1}}
\newcommand{\cororef}[1]{\ref{coro:#1}}
\newcommand{\proporef}[1]{\ref{propo:#1}}
\newcommand{\lemref}[1]{\ref{lem:#1}}
\newcommand{\eqnref}[1]{\eqref{eq:#1}}
\newcommand{\alglab}[1]{\label{alg:#1}}
\newcommand{\algref}[1]{\ref{alg:#1}}
\newcommand{\seclab}[1]{\label{sect:#1}}
\newcommand{\secref}[1]{\ref{sect:#1}}
\newcommand{\tabref}[1]{\ref{tab:#1}}

\renewcommand{\u}{u}
\newcommand{\uk}{\u^k}
\newcommand{\upk}{\u^{+k}}
\newcommand{\ukp}{\u^{k+1}}
\newcommand{\umkp}{\u^{-k+1}}
\newcommand{\umk}{\u^{-k}}
\renewcommand{\v}{v}
\newcommand{\vk}{{\v}^k}
\newcommand{\vkp}{{\v}^{k+1}}
\newcommand{\un}{\u_h}
\renewcommand{\e}{e}
\newcommand{\w}{w}
\renewcommand{\wb}{\mb{\w}}
\newcommand{\J}{J}
\newcommand{\Jb}{\mb{J}}
\newcommand{\q}{q}
\renewcommand{\r}{r}
\newcommand{\qh}{\hat{\q}}
\newcommand{\qs}{\q^*}
\newcommand{\qht}{\tilde{\qh}}
\newcommand{\qbar}{\overline{\q}}
\newcommand{\qb}{{\mb{\q}}}
\newcommand{\sig}{{\sigma}}
\newcommand{\thetah}{{\hat{\theta}}}
\newcommand{\thetas}{{\theta^*}}
\newcommand{\thetahn}{{\thetah_h}}
\newcommand{\sign}[1]{\text{ sgn}\!\LRp{#1}}
\newcommand{\sigh}{\hat{\sig}}
\newcommand{\sigs}{\sig^*}
\newcommand{\sigk}{\sig^{k}}
\newcommand{\sigpk}{\sig^{+k}}
\newcommand{\sigkp}{\sig^{k+1}}
\newcommand{\sigmkp}{\sig^{-k+1}}
\newcommand{\sigmk}{\sig^{-k}}
\newcommand{\sigb}{{\bs{\sig}}}
\newcommand{\sigbk}{\sigb^k}
\newcommand{\sigbkp}{\sigb^{k+1}}
\newcommand{\sigbn}{\sigb_h}
\newcommand{\sigbs}{\sigb^*}
\newcommand{\taub}{{\bs{\tau}}}
\newcommand{\taustar}{\tau_*}
\newcommand{\kappab}{{\bs{\kappa}}}
\newcommand{\gamb}{{\bs{\gamma}}}
\newcommand{\ut}{\tilde{\u}}
\newcommand{\sigbt}{\tilde{\sigb}}
\newcommand{\vt}{\tilde{\v}}
\newcommand{\taubt}{\tilde{\taub}}
\newcommand{\ub}{\mb{\u}}
\newcommand{\ubl}{\mb{\u}_{\lambda}}
\newcommand{\ubf}{\mb{\u}_{f}}
\newcommand{\ubp}{\mb{\u}^+}
\newcommand{\ubm}{\mb{\u}^-}
\newcommand{\vb}{\mb{\v}}
\newcommand{\um}{\u^-}
\newcommand{\up}{\u^+}
\newcommand{\ubk}{\ub^k}
\newcommand{\ubkp}{\ub^{k+1}}

\newcommand{\m}{{m}}

\newcommand{\ud}{{\dot{\u}}}
\newcommand{\vd}{{\dot{\v}}}
\newcommand{\sigbd}{{\dot{\sigb}}}
\newcommand{\taubd}{\dot{\taub}}

\renewcommand{\d}{{d}}
\newcommand{\R}{{\mathbb{R}}}
\newcommand{\kap}{\kappa}
\newcommand{\F}{\mb{F}}
\newcommand{\f}{f}
\newcommand{\g}{g}
\newcommand{\fb}{\mb{\f}}
\newcommand{\gb}{\mb{\g}}
\newcommand{\gD}{g_D}
\newcommand{\pOmega}{\partial \Omega}
\newcommand{\K}{T}
\newcommand{\Kp}{\K^+}
\newcommand{\Km}{\K^-}
\newcommand{\pK}{{\partial \K}}
\newcommand{\pKin}{{\pK^{\text{in}}}}
\newcommand{\pKout}{{\pK^{\text{out}}}}
\newcommand{\Kj}{{\K_j}}
\newcommand{\Gh}{{\mc{E}_h}}
\newcommand{\Gho}{{\mc{E}_h^o}}
\newcommand{\Ghb}{{\mc{E}_h^{\partial}}}
\newcommand{\Poly}{{\mc{P}}}
\newcommand{\D}{{D}}
\newcommand{\p}{{p}}
\newcommand{\pres}{{q}}
\newcommand{\presl}{{q}_{\lambda}}
\newcommand{\presf}{{q}_{f}}
\newcommand{\Ts}{{\mb{T}}}
\newcommand{\Vb}{{\mb{V}}}
\newcommand{\V}{{V}}
\newcommand{\Lam}{{\Lambda}}
\newcommand{\Lamb}{\bs{\Lambda}}
\newcommand{\mub}{\bs{\mu}}
\newcommand{\lambdab}{\bs{\lambda}}
\newcommand{\Thn}{\mathcal{T}_h}
\newcommand{\Th}{{\Ts_h}}
\newcommand{\Vh}{{\V_h}}
\newcommand{\Vbh}{{\Vb_h}}
\newcommand{\Lamh}{{\Lam_h}}
\newcommand{\Lambh}{{\Lamb_h}}
\newcommand{\Lamho}{{\overline{\Lam}_h}}
\newcommand{\Lambht}{{\Lamb_h^t}}
\newcommand{\Ltwo}{{L^2\LRp{\Omega}}}
\newcommand{\Ltwoe}{{L^2\LRp{\e}}}
\newcommand{\Ltw}{{L^2\LRp{\K}}}
\newcommand{\VhK}{{\V_h\LRp{\K}}}
\newcommand{\ThK}{{\Ts_h\LRp{\K}}}
\newcommand{\VbhK}{{\Vb_h\LRp{\K}}}
\newcommand{\Lamhe}{{\Lam_h\LRp{\e}}}
\newcommand{\Lambhe}{{\Lamb_h\LRp{\e}}}
\renewcommand{\L}{\mc{L}}
\renewcommand{\D}{\mc{D}}
\newcommand{\T}{\mc{T}}
\newcommand{\Tt}{\tilde{\T}}
\newcommand{\A}{\mb{A}}
\newcommand{\B}{\mb{B}}
\renewcommand{\D}{\mb{D}}
\newcommand{\Am}{\A^-}
\newcommand{\Ap}{\A^+}
\newcommand{\Aa}{\snor{\A}}
\newcommand{\Rb}{\mb{R}}
\newcommand{\C}{\mb{C}}
\renewcommand{\S}{\mb{S}}
\newcommand{\Sm}{\S^{-}}
\newcommand{\Sp}{\S^{+}}
\newcommand{\Spm}{\S^{\pm}}

\newcommand{\Lte}{{L^2\LRp{\Gh}}}
\newcommand{\n}{\mb{n}}
\newcommand{\nm}{\n^-}
\newcommand{\np}{\n^+}
\newcommand{\uh}{{\hat{\u}}}
\newcommand{\uhk}{{\uh}^k}
\newcommand{\uhkp}{{\uh}^{k+1}}
\newcommand{\us}{{\u^*}}
\newcommand{\uhn}{{\uh_h}}
\newcommand{\ubh}{{\hat{\ub}}}
\newcommand{\ubs}{{\ub^*}}
\newcommand{\ubht}{{\tilde{\ubh}}}
\newcommand{\uht}{{\tilde{\uh}}}
\newcommand{\uhtn}{{\uht_h}}
\newcommand{\uhd}{{\dot{\uh}}}
\newcommand{\ubhk}{\ubh^k}
\newcommand{\ubhkp}{\ubh^{k+1}}
\newcommand{\sigbh}{{\hat{\sigb}}}
\newcommand{\Fh}{{\hat{\F}}}
\newcommand{\Fs}{\F^*}
\renewcommand{\c}{{c}}

\newcommand{\zb}{\mb{z}}
\renewcommand{\sb}{\mb{s}}
\newcommand{\rb}{\mb{r}}
\newcommand{\betab}{\bs{\beta}}
\newcommand{\veps}{{\varepsilon}}
\newcommand{\vepsb}{\bs{\veps}}
\newcommand{\Hb}{\mb{H}}
\newcommand{\Hbt}{\Hb^t}
\newcommand\E{\mathcal{E}}
\newcommand\Eh{\mathcal{E}_h}
\newcommand\Ehi{\mathcal{E}_{h,i}}
\newcommand\Ehiint{\mathcal{E}_{h,i}^\mathrm{int}}
\newcommand\EhG{\mathcal{E}_{h}^{\Gamma}}
\newcommand\EhiG{\mathcal{E}_{h,i}^{\Gamma}}
\newcommand\EhijG{\mathcal{E}_{h,i,j}^{\Gamma}}
\newcommand\EhjiG{\mathcal{E}_{h,j,i}^{\Gamma}}
\newcommand{\Eb}{\mb{E}}
\newcommand{\Ebt}{\Eb^t}
\newcommand{\hb}{\mb{h}}
\newcommand{\eb}{\mb{e}}
\newcommand{\Hbh}{\hat{\mb{H}}}
\newcommand{\Ebh}{\hat{\mb{E}}}
\newcommand{\Hbs}{{\mb{H}^*}}
\newcommand{\Ebs}{{\mb{E}^*}}
\newcommand{\Ebht}{\Ebh^t}
\newcommand{\Hbht}{\Hbh^t}
\newcommand{\Lb}{\mb{L}}
\newcommand{\Gb}{\mb{G}}
\newcommand{\Lbh}{\hat{\Lb}}
\newcommand{\Lbs}{\Lb^*}
\newcommand{\Ib}{\mb{I}}
\newcommand{\I}{I}
\newcommand{\Q}{Q}
\newcommand{\Sigb}{\bs{\Sigma}}
\newcommand{\Piu}{\Pi_\u}
\newcommand{\Pisigb}{\Pi_\sigb}
\newcommand{\Z}{Z}
\newcommand{\Y}{Y}
\newcommand{\rhou}{\rho\u}
\newcommand{\rhov}{\rho\v}
\renewcommand{\P}{P}
\newcommand{\h}{H}

\newcommand{\mygreen}[1]{{\color[rgb]{0,.65,0} #1}}
\newcommand{\mywhite}[1]{{\color[rgb]{1.0,1.0,1.0} #1}}
\newcommand{\myred}[1]{{\color[rgb]{0.65,0.0,0.0} #1}}
\newcommand{\myblue}[1]{{\color[rgb]{0.0,0.0,1.0} #1}}
\newcommand{\mygray}[1]{{\color[rgb]{.15,.35,.60} #1}}
\newcommand{\mythemec}[1]{{\color[RGB]{51,108,121} #1}}
\newcommand{\mydarkred}[1]{{\color[rgb]{.6,.1,.1} #1}}
\newcommand{\mydarkblue}[1]{{\color[rgb]{.1,.1,.9} #1}}
\newcommand{\mygreenback}[1]{{\color[rgb]{.75,1,.75} #1}}
\newcommand{\myorange}[1]{{\color[rgb]{.76,.39,.13} #1}}
\newcommand{\mygrass}[1]{{\color[rgb]{.19,.64,.13} #1}}
\newcommand{\mysierp}[1]{{\color[RGB]{209,28,209} #1}}

\newcommand{\tanbui}[2]{\textcolor{blue}{#1} \textcolor{red}{#2}}

\newcommand{\vel}{\bs{\vartheta}}
\newcommand{\vels}{\vel^*}
\newcommand{\vele}{\vel^e}
\newcommand{\velh}{\hat{\vel}}
\newcommand{\velk}{\vel^k}
\newcommand{\velkp}{\vel^{k+1}}
\newcommand{\vhk}{\hat\v^{k}}

\newcommand{\phis}{\phi^*}
\newcommand{\phie}{\phi^e}
\newcommand{\phih}{\hat{\phi}}
\newcommand{\phihk}{\phih^k}
\newcommand{\phik}{\phi^k}
\newcommand{\phikp}{\phi^{k+1}}

\maketitle

\begin{abstract}
    We consider a standard elliptic partial differential equation and
    propose a geometric multigrid algorithm based on
    Dirichlet-to-Neumann (DtN) maps for hybridized high-order finite
    element methods.  The proposed unified approach is
     applicable to any locally conservative hybridized finite
    element method including multinumerics with different
    hybridized methods in different parts of the domain.
    For these
    methods, the linear system involves only the unknowns residing on
    the mesh skeleton, and constructing intergrid transfer operators is
    therefore not trivial.
The key to our geometric multigrid algorithm is the physics-based energy-preserving intergrid
transfer operators which depend only on the fine scale DtN maps. Thanks to these operators, we completely avoid upscaling 
of parameters and no information regarding subgrid physics is explicitly required on coarse meshes. Moreover, our algorithm is 
agglomeration-based
and can straightforwardly handle unstructured meshes.
We perform extensive numerical studies with hybridized mixed methods,
hybridized discontinuous Galerkin method, weak Galerkin method, and a hybridized version of interior penalty discontinuous Galerkin methods 
on a range of elliptic problems including  subsurface flow through  highly heterogeneous porous media. We compare the performance of different smoothers and analyze the effect of
stabilization parameters on the scalability of the multigrid algorithm. 
\end{abstract}

\pagestyle{myheadings} \thispagestyle{plain} \markboth{T.~Wildey, S.~Muralikrishnan, T.~Bui-Thanh}{Unified Geometric Multigrid}





\section{Introduction}
\label{sec_int}

Hybridization of finite element methods was first introduced in 1965 \cite{fraeijs1965displacement}
for solving linear elasticity problems. Hybridized finite element
methods have come a long way since then and a vast amount of research
has been done over the past few decades 
(see,
e.g. \cite{roberts1991mixed,cockburn2004characterization,Gir_Sun_Whee_Yot_DG_MFE_mort_08,CockburnGopalakrishnanLazarov:2009:UHO,brezzi2012mixed}). Hybridized
methods offer a number of significant advantages over the original ones, and some
of which are: 1) the resulting linear system can be significantly
smaller and sparser 
\cite{CockburnGopalakrishnanLazarov:2009:UHO,bui2016construction}; 2)
$hp$-adaptivity is natural for hybridized
methods due to the already available skeletal space; 3) it offers
a natural way to couple different numerical methods in different parts
of the domain (multinumerics) and hence exploiting their individual
strengths; and 4) for certain problems the trace unknowns can be used to  post-process the solution to obtain
super-convergence \cite{arnold1985mixed,CockburnGopalakrishnanLazarov:2009:UHO}.

The main challenge facing  hybridized methods is however the construction of scalable solvers for the linear system
involving only trace unknowns on the mesh skeleton. 
Over the past 30 years, a
tremendous amount of research has been devoted to the convergence of
multigrid methods for such linear systems, both as iterative methods
and as preconditioners for Krylov subspace methods.  Optimal
convergence with respect to the number of unknowns is usually obtained
under mild elliptic regularity assumptions
\cite{bramble1993multigrid,Bramble94uniformconvergence,bpx1991}.
Multigrid algorithms have been developed for mortar domain
decomposition methods \cite{339353,YotovMultigrid}.  Several multigrid
algorithms have been proposed for hybridized mixed finite element
methods \cite{Bren_MG_MFE_92, Chen_equiv_96}.  Most of them are based
on an equivalence between the interface operator and a nonconforming
finite element method,  
and for these methods optimal convergence has
already been established \cite{Braess:1990,
  Brenner:nonconfmg}. Multigrid algorithms based on restricting the
trace (skeletal) space to linear continuous finite element space has been
proposed for hybridized mixed methods
\cite{Gopalakrishnan09aconvergent}, hybridized discontinuous Galerkin
methods \cite{cockburn2014multigrid} and weak Galerkin methods
\cite{chen2015auxiliary}. These algorithms fall under the non-inherited category, i.e., the coarse scale
operators do not inherit all the properties of the fine scale ones.
To the best of our knowledge, no multigrid
algorithm has been developed for the case of multinumerics.

The objective of this work is to develop a multigrid algorithm 
that applies to both hybridized formulations and multinumerics.
The algorithm applies to both structured and unstructured grids. At
the heart of our approach is the energy-preserving intergrid transfer
operators which are a function of only the fine scale DtN maps (to be discussed in details in section \secref{multigrid_algorithm}).  As
such they avoid any explicit upscaling of parameters, and at the same
time allow for the use of multinumerics throughout the domain. The
multigrid algorithm presented in this paper thus differs from the
existing approaches in that the Galerkin coarse grid operator is a
discretized DtN map on every level. 



This paper is organized as follows. Section
\secref{hybrid_methods} introduces the model problem, notation, and
 hybridized methods considered in this paper. In Section
\secref{multigrid_algorithm}, we define the necessary ingredients for
our geometric multigrid algorithm, that is, the coarsening strategy, the
intergrid transfer operators, the local correction operator, and the
smoothing operator. These operators are then used to define the multigrid algorithm. 
Section \secref{numerical} presents several
numerical examples to study the robustness of the proposed
algorithm for different hybridized methods, smoothers and test
cases. Finally, Section \secref{conclusion} summarizes our
findings and discusses  future research directions.


\section{Model Problem, Notation and Hybridized methods}\seclab{hybrid_methods}
In this section we will first introduce the notations and discuss
some of the locally conservative hybridized methods
proposed in the recent years for the second order elliptic
equation.

Consider the following second order elliptic equation
\begin{subequations}
  \label{primal_elliptic}
\begin{align}
    -\Div\LRp{{\bf K} \Grad \pres} &= f, \quad &\text{in} \quad \Omega, \\
    \pres&=g_{D},  \quad  &\text{on} \quad \pOmega.
\end{align}
\end{subequations}
where $\Omega$ is an open, bounded, and connected subset of $\mathbb{R}^d$, with $d\in \LRc{2,3}$. Here, $\tK$ is a
symmetric, bounded, and uniformly positive definite tensor, $f \in L^2(\Omega)$, and
$g_D \in H^{3/2}(\partial \Omega)$. 
Let $\Thn$ be a conforming partition of $\Omega$ into
$\Nel$ non-overlapping elements $\Kj$, $j = 1, \hdots, \Nel$, with
Lipschitz boundaries such that $\Thn := \cup_{j=1}^\Nel \Kj$ and
$\overline{\Omega} = \overline{\Thn}$. The mesh size $h$ is defined as $h
:= \max_{j\in \LRc{1,\hdots,\Nel}}\text{diam}\LRp{\Kj}$. We denote the
skeleton of the mesh by $\Gh := \cup_{j=1}^\Nel \pK_j$:
the set of all (uniquely defined) interfaces $\e$ between elements. We conventionally identify $\nm$ as the outward 
normal vector on the boundary $\pK$ of element $\K$ (also denoted as $\Km$) and $\np = -\nm$ as the outward normal vector of the boundary of a neighboring element (also denoted as $\Kp$). Furthermore, we use $\n$ to denote either $\nm$ or $\np$ in an expression that is valid for both cases, and this convention is also used for other quantities (restricted) on
a face $\e \in \Gh$.


For simplicity, we define $\LRp{\cdot,\cdot}_\K$ as the
$L^2$-inner product on a domain $\K \subset \R^\d$ and
$\LRa{\cdot,\cdot}_\K$ as the $L^2$-inner product on a domain $\K$ if
$\K \subset \R^{\d-1}$. We shall use $\nor{\cdot}_{\K} :=
\nor{\cdot}_{\Ltw}$ as the induced norm for both cases.
Boldface lowercase letters are conventionally used for vector-valued functions and in that
case the inner product is defined as $\LRp{\ub,\vb}_\K :=
\sum_{i=1}^m\LRp{\ub_i,\vb_i}_\K$, and similarly $\LRa{\ub,\vb}_\K :=
\sum_{i = 1}^m\LRa{\ub_i,\vb_i}_\K$, where $\m$ is the number of
components ($\ub_i, i=1,\hdots,\m$) of $\ub$.  Moreover, we define
$\LRp{\ub,\vb}_{\Thn} := \sum_{\K\in \Thn}\LRp{\ub,\vb}_\K$ and
$\LRa{\ub,\vb}_\Gh := \sum_{\e\in \Gh}\LRa{\ub,\vb}_\e$ whose
induced norms are clear, and hence their definitions are
omitted. We  employ boldface uppercase letters, e.g. $\tK$, to
denote matrices and tensors. 
We denote by $\Poly^\p\LRp{\K}$ the space of polynomials of degree at
most $\p$ on a domain $\K$. We use the terms ``skeletal unknowns",
``trace unknowns" and ``Lagrange multipliers" interchangeably and they all refer to the unknowns 
 on the mesh skeleton.
 
 First, we cast equation \eqref{primal_elliptic} into the following first-order or mixed form:
 \begin{subequations}
   \label{mixed_elliptic}
     \begin{align}\label{mixed_eq1}
         \ub&=-{\bf K}\Grad \pres \quad &\text{ in }  \Omega,\\\label{mixed_eq2}
         \Div \ub &= f \quad &\text{ in } \Omega,\\\label{mixed_boundary}
    \pres &= g_{D} \quad &\text {on } \pOmega.
\end{align}
\end{subequations}

The hybrid mixed DG method or hybridized DG (HDG) method for the
discretization of equation \eqref{mixed_elliptic} is defined as
\begin{subequations}
  \label{HDG}
\begin{align}\label{HDG_local1}
    \LRp{{\bf K}^{-1}\ub,\vb}_\K -\LRp{\pres,\Div \vb}_\K + \LRa{\lambda,\vb \cdot \n}_\pK &= 0, \\\label{HDG_local2}
    -\LRp{\ub,\Grad \w}_\K + \LRa{\ubh \cdot \n,\w}_\pK &= \LRp{f,w}_\K, \\
    \LRa{\jump{\ubh \cdot \n},\mu}_\e&=0,
    \label{conservation}
\end{align}
\end{subequations}
where the numerical flux $\ubh \cdot \n$ is given by
\begin{equation}
    \ubh \cdot \n = \ub \cdot \n + \tau(\pres - \lambda).
    \label{HDG_flux}
\end{equation}

For simplicity, we have ignored the fact that the equations \eqref{HDG_local1}, \eqref{HDG_local2} and \eqref{conservation} must hold for all test functions $\vb \in \VbhK$, $\w \in W_{h}\LRp{\K}$, and $\mu \in M_{h}\LRp{\e}$, respectively (this is implicitly understood throughout the paper), where $\Vbh$, $W_h$ and $M_h$ are defined as
\begin{align*}
\Vbh\LRp{\Thn} &= \LRc{\vb \in \LRs{L^2\LRp{\Thn}}^d:
  \eval{\vb}_{\K} \in \LRs{\Poly^\p\LRp{\K}}^d, \forall \K \in \Thn}, \\
 W_h\LRp{\Thn} &= \LRc{\w \in L^2\LRp{\Thn}:
  \eval{\w}_{\K} \in \Poly^\p\LRp{\K}, \forall \K \in \Thn}, \\   
M_h\LRp{\Gh} &= \LRc{\lambda \in \Lte:
  \eval{\lambda}_{\e} \in \Poly^\p\LRp{\e}, \forall \e \in \Gh},
\end{align*}
and similar spaces  $\VbhK$, $W_h\LRp{\K}$ and $M_h\LRp{\e}$ on $\K$ and $\e$ can be defined by replacing $\Thn$ with
$\K$ and $\Gh$ with $\e$, respectively.


Next, we consider the hybridized interior penalty DG (IPDG) schemes posed in the primal form. To that end, we test equation \eqref{mixed_eq1} with $\vb=\Grad \w$  and then integrate by parts twice the terms on the right hand side to obtain
\begin{equation*}
    \LRp{\ub,\Grad \w}_\K =-\LRp{{\bf K}\Grad\pres,\Grad \w}_\K + \LRa{(\pres-\lambda),{\bf K}\Grad \w \cdot \n}_\pK.
\end{equation*}
Substituting this in equation \eqref{HDG_local2} gives
\begin{equation}
\LRp{{\bf K}\Grad\pres,\Grad \w}_\K - \LRa{(\pres-\lambda),{\bf K}\Grad \w \cdot \n}_\pK + \LRa{\ubh \cdot \n,\w}_\pK = \LRp{f,w}_\K. 
\label{IPDG_flux_form}
\end{equation}
For IPDG schemes, the numerical flux takes the form
\begin{equation}
    \ubh \cdot \n = -{\bf K}\Grad \pres \cdot \n + \tau(\pres - \lambda),
    \label{IPDG_flux}
\end{equation}
and it is required to satisfy the conservation condition \eqref{conservation}.
Substituting the numerical flux \eqref{IPDG_flux} in \eqref{IPDG_flux_form} we get the following primal form of the hybridized IPDG scheme
\begin{multline}
 \LRp{{\bf K}\Grad\pres,\Grad \w}_\K - \LRa{(\pres-\lambda),{\bf K}\Grad \w \cdot \n}_\pK - \LRa{{\bf K}\Grad \pres \cdot \n,\w}_\pK
  \\  +\LRa{\tau(\pres-\lambda),\w}_\pK = \LRp{f,w}_\K. 
\label{SIPG_H}
\end{multline}
This scheme is called the hybridized symmetric IPDG scheme (SIPG-H)
and has been studied in
\cite{Gir_Sun_Whee_Yot_DG_MFE_mort_08,CockburnGopalakrishnanLazarov:2009:UHO,waluga2012analysis}.
In order to include the other hybridized IPDG schemes, we can generalize \eqref{SIPG_H} to
\begin{multline}
    \LRp{{\bf K}\Grad\pres,\Grad \w}_\K - s_{f}\LRa{(\pres-\lambda),{\bf K}\Grad \w \cdot \n}_\pK - \LRa{{\bf K}\Grad \pres \cdot \n,\w}_\pK
    \\ +\LRa{\tau(\pres-\lambda),\w}_\pK = \LRp{f,w}_\K. 
\label{IPG_H}
\end{multline}
where $s_{f} \in \LRc{-1,0,1}$. The scheme corresponding to $s_{f}=-1$ is called the
hybridized non-symmetric IPDG scheme (NIPG-H) and the one with
$s_{f}=0$ is called the hybridized incomplete IPDG scheme (IIPG-H)
\cite{Gir_Sun_Whee_Yot_DG_MFE_mort_08}. 
It is interesting to note that in the HDG scheme, if we
do not integrate by parts equation \eqref{HDG_local1}  and
substitute $\ub=-{\bf K}\Grad\pres$ in equations \eqref{HDG_local2}
and \eqref{HDG_flux} we obtain the IIPG-H scheme.  On the other hand, if
we do not integrate by parts equation \eqref{HDG_local2} and take $\ubh\cdot\n=\ub\cdot\n$ in equation
\eqref{conservation} we obtain the hybridized mixed methods as follows
\begin{subequations}
\begin{align}\label{mixed_local1}
    \LRp{{\bf K}^{-1}\ub,\vb}_\K -\LRp{\pres,\Div \vb}_\K + \LRa{\lambda,\vb \cdot \n}_\pK &= 0, \\\label{mixed_local2}
    \LRp{\Div \ub,\w}_\K &= \LRp{f,w}_\K,
\end{align}
\end{subequations}
and we need to enforce the conservation condition $\LRa{\jump{\ub \cdot \n},\mu}_\e=0$.
 However, in
this case we can no longer choose the same solution order for both $\ub$ and $\pres$ due to
the lack of stabilization. Indeed, inf-sup stable mixed finite element spaces have
to be chosen for $\Vbh$ and $W_h$. For the details, we refer the readers to \cite{cockburn2004characterization} for the hybridized Raviart-Thomas (RT-H) and the hybridized Brezzi-Douglas-Marini (BDM-H) mixed finite element spaces. 

We note in passing that the hybridized weak Galerkin mixed finite
element methods introduced in \cite{mu2016hybridized} follows similar
hybridization as done in the hybridized mixed methods but with a different
choice for local spaces accompanied with a stabilization term. In many
cases one can show that HDG and weak Galerkin methods coincide, and we
do not attempt to distinguish them in this paper.

The common solution procedure for all of these hybridized methods can
now be described. First, we express the local volume unknowns $\ub$
and/or $\pres$, element-by-element, as a function of the skeletal
unknowns $\lambda$. Then, we use the conservation condition to
construct a global linear system involving only the skeletal
unknowns. Once they are solved for, the local volume
unknowns can be recovered in an element-by-element fashion completely
independent of each other. The main advantage of this Schur complement
approach is that, for high-order methods, the global trace system is
much smaller and sparser compared to the linear system for the volume
unknowns
\cite{CockburnGopalakrishnanLazarov:2009:UHO,bui2016construction}.
The question that needs to be addressed is how to solve the trace system efficiently. In the next section we
develop a geometric multigrid algorithm to provide an answer.

\section{Geometric multigrid algorithm based on DtN maps}\seclab{multigrid_algorithm}
For all of the hybridized methods described in the previous section, the resulting linear systems for the skeletal unknowns $\lambda$, in operator form, can be written as
\begin{equation}
    \label{trace_linear_system}
    A \lambda = g.
\end{equation}
The well-posedness of the trace system \eqref{trace_linear_system} for the 
hybridized methods discussed in the previous section has been shown in 
\cite{CockburnGopalakrishnanLazarov:2009:UHO,Gir_Sun_Whee_Yot_DG_MFE_mort_08,mu2016hybridized}.

The concept of a Dirichlet-to-Neumann (DtN) map is essential
    to our approach, so we briefly describe it here. A map $A$ is
    called a DtN map/operator if it maps Dirichlet data on a
    domain boundary to Neumann data. It is a particular type of
    Poincar\'e-Steklov operator, which contains a large class of operators
    which map one type of boundary condition to another for elliptic
    PDEs. For hybridized methods in equation \eqref{HDG}, the
    trace variable $\lambda$ plays the role as the flux for the
    auxiliary equation \eqref{HDG_local1} and hence the operator $A$
    in equation \eqref{trace_linear_system} is a DtN operator which
    maps Dirichlet boundary data of the original PDE
    \eqref{primal_elliptic} to Neumann data $\lambda$. In
    finite-dimension, the Schur complement of the
    linear system with volume unknowns condensed out is also known as discrete DtN map. We refer the
    readers to \cite{quarteroni1999domain} for more information.

To define our multigrid algorithm 
we start with a sequence of partitions of the mesh $\Thn$:
$$ \T_1,\T_2,\ldots,\T_N = \Thn,$$ where each $\T_k$ contains
$N_{T_k}$ closed (not necessarily convex, except on the finest level
  where each $N_{T_k}$ is in fact some element $\K_j$ in $\Thn$)
macro-elements consisting of
unions of macro-elements from $\T_{k+1}$.
Associated with each partition, we define interface grids $
\E_{1},\E_{2},\ldots,\E_{N} = \Eh, $ where each $e \in \E_k$ is the
intersection of two macro-elements in $\T_k$. Each partition $\E_{k}$
is in turn associated with a skeletal (trace)
space $M_{k}$ (to be defined in Section \secref{c_strategy}).

\begin{remark}
We note that if each $\T_k$ consists of simplices or parallelopipeds,
then each $M_k$ is straightforward to define.  In general, however, the
macro-elements need not be one of these standard shapes, and we need to
 define a parameterization of each macro-edge to define $M_k$.
While this is certainly nontrivial, this calculation only
needs to be performed one time for each mesh and can be reused for
many simulation or time steps as long as the macro-elements
(agglomeration) do not change.  For all of the results presented in
Section \secref{numerical}, the time required to calculate these
parameterizations is negligible in comparison with the total
simulation time.
\end{remark}




We are now in the position to introduce necessary ingredients to define our geometric multigrid algorithm.

\subsection{Coarsening strategy}\seclab{c_strategy}
In the multigrid algorithm we first coarsen in $\p$, followed by the coarsening in $h$. We explain
these two coarsening procedures in detail as follows.

\subsubsection{$\p$-coarsening} In our $\p$-coarsening strategy, we restrict the solution order $\p$
on the finest level $N$ to $\p=1$ on level $N-1$ with the same number of elements. In this case the
Lagrange multiplier spaces become
\[ M_k = 
\begin{cases}
    \LRc{\eta\in \Poly^{\p}(e), \forall e \in \E_k} &\quad \text{for} \quad k=N,\\
    \LRc{\eta\in \Poly^1(e), \forall e \in \E_k} &\quad \text{for} \quad k=1,2,\cdots,N-1.
\end{cases}
\]
Let us comment on one subtle point.  It would be ideal if we could use
$\Poly^0$ in coarser levels because it is trivial to define
piecewise constant functions along arbitrary macro-edges (2D) or
macro-faces (3D), i.e., no parameterization is needed. However, multilevel algorithms using restriction
and prolongation operators based on piecewise constant interpolation
typically have been shown not to perform well \cite{ccfd:1996,kwak:552}.

For high-order methods, the numerical examples in Section~\secref{numerical} indicate that 
our strategy gives scalable results in solution order $\p$ when
strong smoothers such as block-Jacobi or Gauss-Seidel is used. This is also observed in
\cite{helenbrook2006application}, for a $\p$-multigrid approach applied to DG methods for Poisson equation using block-Jacobi smoother. Other
$\p$-multigrid strategies, such as $\p_{k-1}=\p_{k}/2$ ($\p_k$ is the solution order on the $k$th level) or
 $\p_{k-1}=\p_{k}-1$, can be straightforwardly incorporated within our approach. 
Our future work will include
theoretical and numerical comparisons among these strategies. 
Once we interpolate to $\p=1$ we carry out the $h$-coarsening as described in the next section.

\begin{figure}[h!t!b!]
  \subfigure[]{
      \includegraphics[width=0.22\columnwidth]{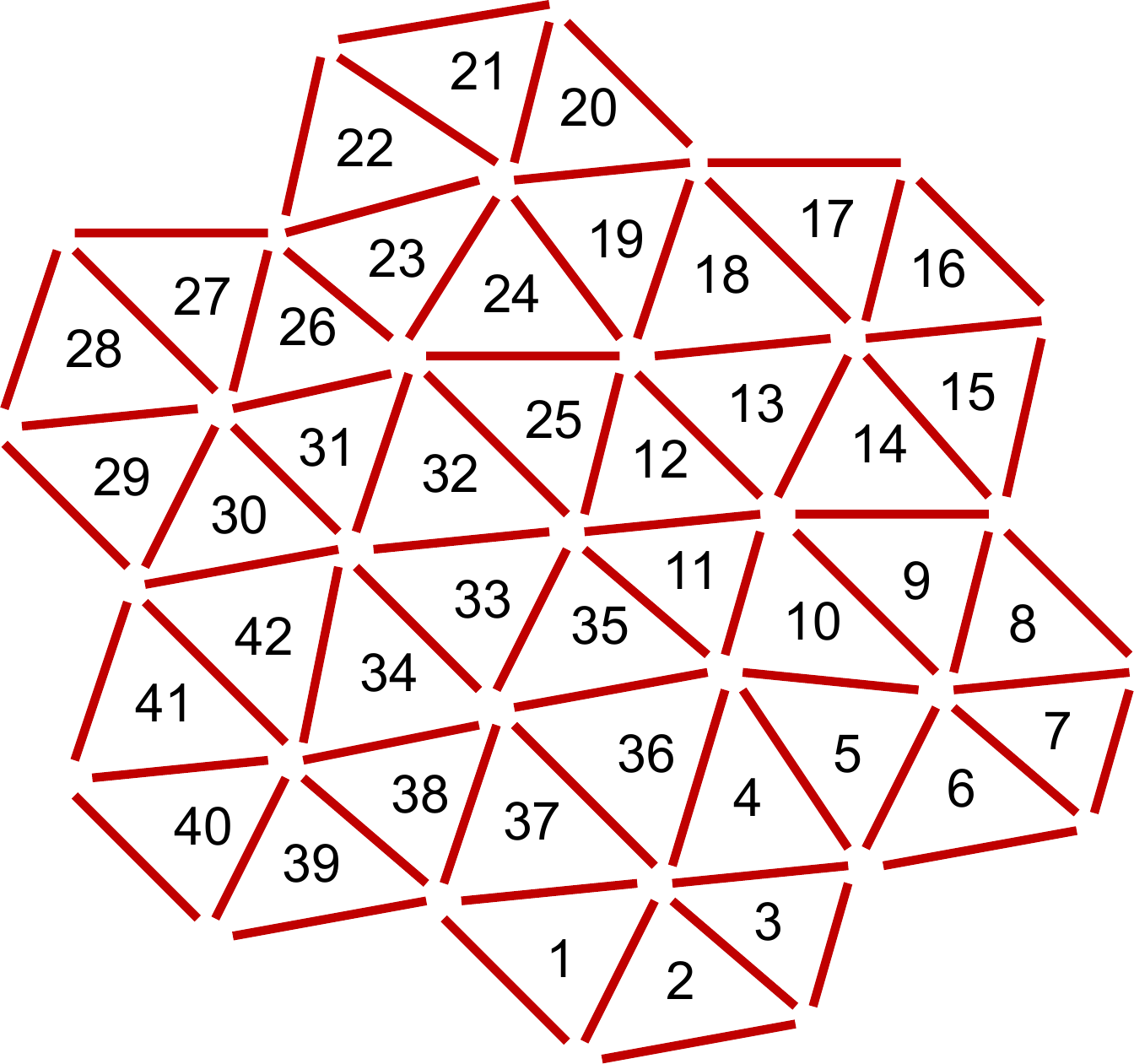}
      \figlab{meshk}
  }
    \subfigure[ ]{
      \includegraphics[width=0.22\columnwidth]{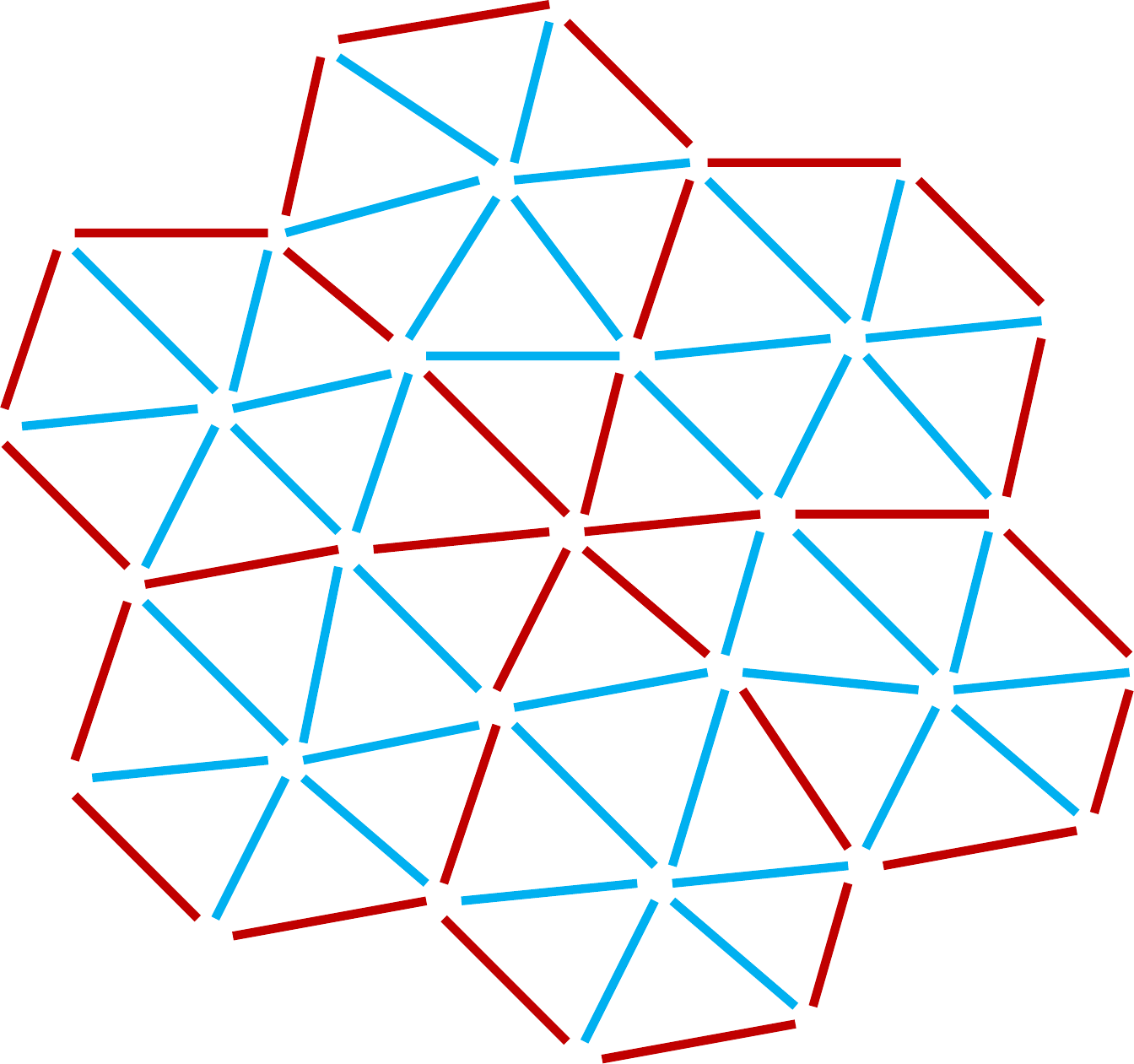}
      \figlab{identify_int_bound}
  }
  \subfigure[]{
      \includegraphics[width=0.22\columnwidth]{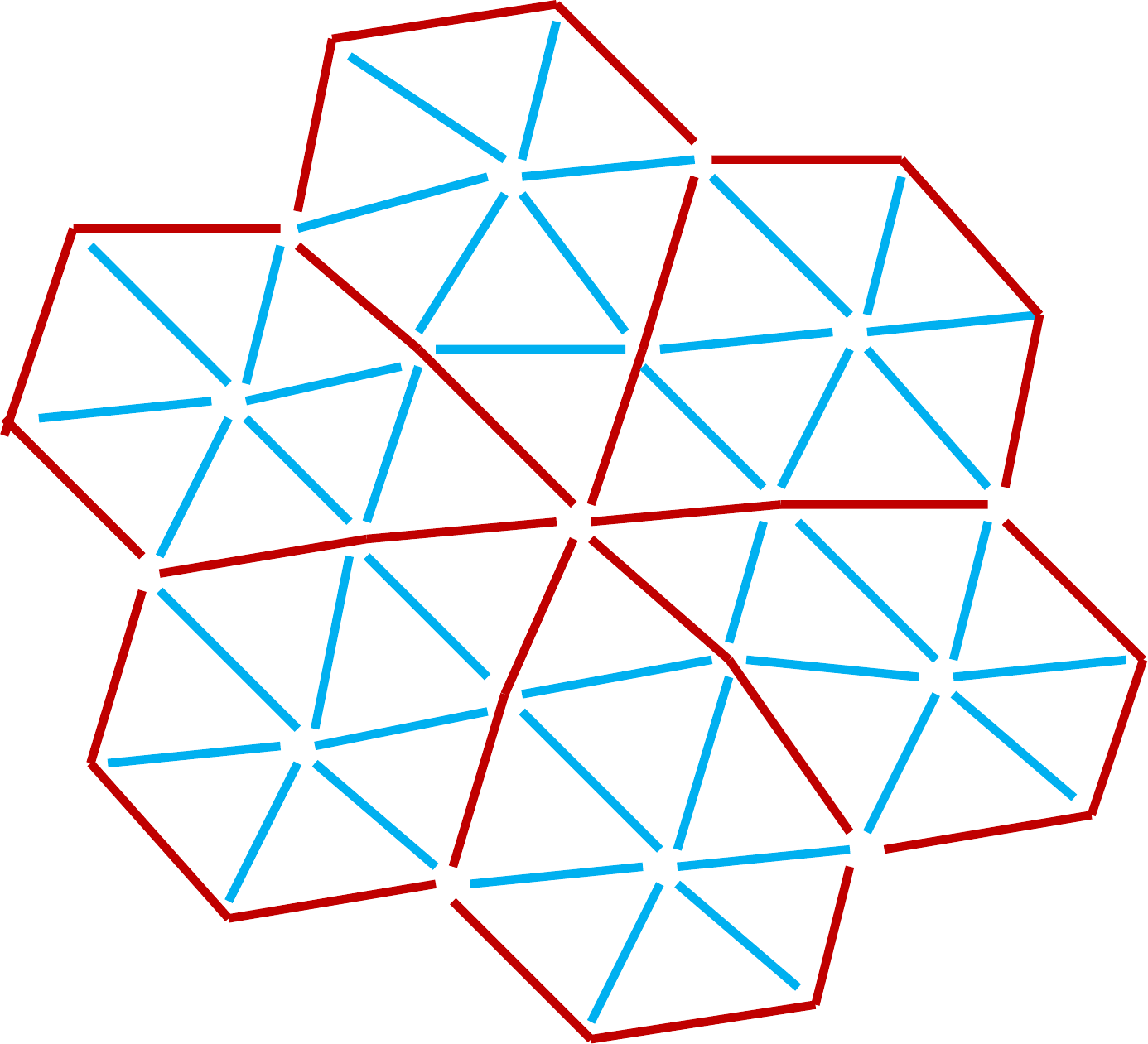}
      \figlab{coarsen_bound}
  }
  \subfigure[]{
      \includegraphics[width=0.22\columnwidth]{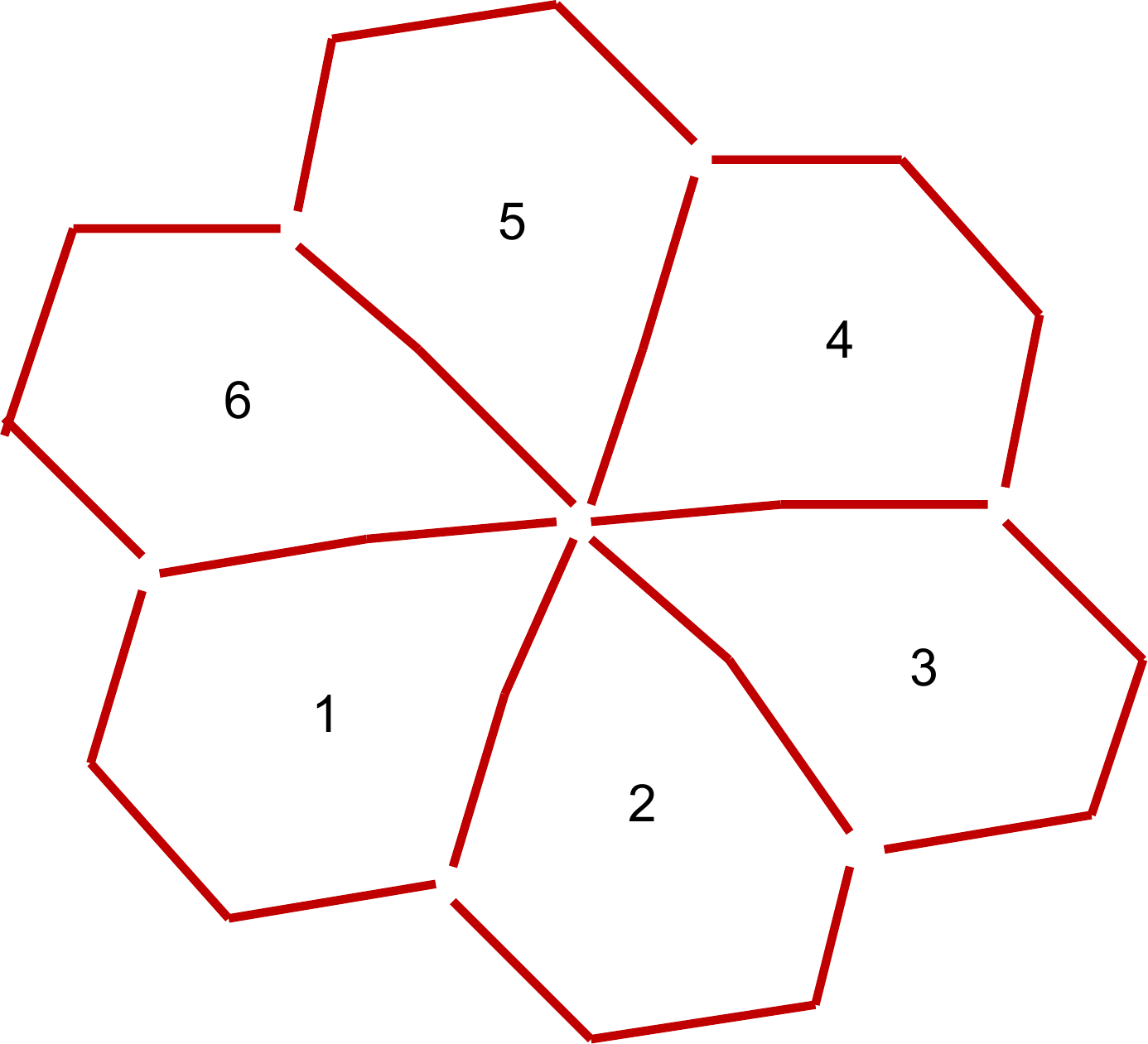}
      \figlab{meshkm1}
  }
    \caption{A demonstration of $h$-coarsening strategy. (a) Level $k$ mesh; (b) An identification of interior $\E_{k,I}$ (blue) and boundary $\E_{k,B}$ (red) edges; (c) Coarsening of boundary edges; and (d) Level $k-1$ mesh after interior edges are statically condensed out. The numbers in (a) and 
    (d) represent the number of (macro) elements in levels $k$ and
    $k-1$, respectively.}
 \vspace{-7mm}
  \figlab{coarsening_strategy}
\end{figure}

\subsubsection{$h$-coarsening} We adopt an agglomeration-based $h$-coarsening, allowing
intergrid transfer and coarse grid operators to be defined using
 the fine scale DtN maps.  By taking this approach, no upscaling
of parameters is required and our coarse operators are discretized
DtN maps at any mesh
level (see Proposition \proporef{DtN}). In Figure~\figref{coarsening_strategy}, we show an
$h$-coarsening strategy between two consecutive levels $k$ and $k-1$. First,
we identify the interior (blue) and boundary (red) edges in Figure \figref{identify_int_bound}. 
We decompose $\E_k = \E_{k,I}\oplus \E_{k,B}$,
where $\E_{k,I}$
consists of edges interior to macro elements in the coarser partition $\T_{k-1}$ and $\E_{k,B}$ contains edges common to their
boundaries. 
We also decompose the trace space $M_k$ on $\E_k$ into two parts $M_{k,I}$ and $ M_{k,B}$ corresponding to $\E_{k,I}$ and $\E_{k,B}$, respectively. Specifically, we require  $M_k = M_{k,I}\oplus M_{k,B}$ such that each $\lambda_k
\in M_k$ can be uniquely expressed as $\lambda_k = \lambda_{k,I} +
\lambda_{k,B}$, where
\[ \lambda_{k,I} = 
\begin{cases}
 \lambda_k,& \text{ on } M_{k,I},\\
 0,& \text{ on } M_{k,B},
\end{cases}
\quad \text{ and } \quad
\lambda_{k,B} = 
\begin{cases}
 0,& \text{ on } M_{k,I},\\
 \lambda_k,& \text{ on } M_{k,B}.
\end{cases}
\]
Given the decomposition $M_k = M_{k,I}\oplus M_{k,B}$, the trace system \eqref{trace_linear_system} at the $k$th level
can be written as 
\begin{equation}
  \label{trace_partition}
  A_k \lambda_k = g_k \Leftrightarrow
        \LRs{
     \begin{array}{cc}
       A_{k,II} & A_{k,IB} \\
         A_{k,BI} & A_{k,BB}
     \end{array}
      }
      \LRs{
      \begin{array}{c}
          \lambda_{k,I} \\
          \lambda_{k,B}
      \end{array}
      }
      =
      \LRs{
      \begin{array}{c}
          g_{k,I} \\
          g_{k,B}
      \end{array}
      }.
\end{equation}
The coarser space $M_{k-1}$ is defined such that $M_{k-1} \subset M_{k,B}$.
This is done by first agglomerating the boundary edges of level $k$ as in Figure \figref{coarsen_bound}, and then statically condensing out the interior edges to obtain the mesh on level $k-1$
in Figure \figref{meshkm1}. The elements in level $k$ are numbered from
$1$ to $42$ in Figure \figref{meshk} and the macro elements in level $k-1$
are numbered from $1$ to $6$ in Figure \figref{meshkm1}

Clearly, for a unstructured mesh the identification of interior and
boundary edges is non-trivial and non-unique. In this paper we use an
ad-hoc approach which is sufficient
for the purpose of demonstrating our proposed algorithm. Specifically,
we first select a certain number of levels $N$ and seed points
$N_{T_1}$. The locations of these seed points are chosen based on the geometry of the domain and the original fine mesh, such that we approximately have equal number of fine mesh elements in each macro-element at level $1$. Based on these
selected seed points, we agglomerate the elements in the finest mesh
($k=N$) to form $N_{T_1}$ macro-elements at level $1$. We then divide each macro-element in level $1$ into four approximately equal macro-elements to form level $2$. This process is recursively repeated to create macro-elements in finer levels $k=3,\cdots,N-1$.
As an illustration we consider the mesh in Figure \ref{fat_box} (this mesh and the corresponding coarsening strategies will be used in numerical example II in Section \secref{exp2}) and in Figures \figref{cs1} and \figref{cs2} we show two different coarsening strategies with $N = 7$ obtained from seven and four seed points. 
\begin{figure}[h!b!t!]
\begin{center}
\includegraphics[trim=3.5cm 11.15cm 2cm 10.65cm,clip=true,width=0.5\textwidth]{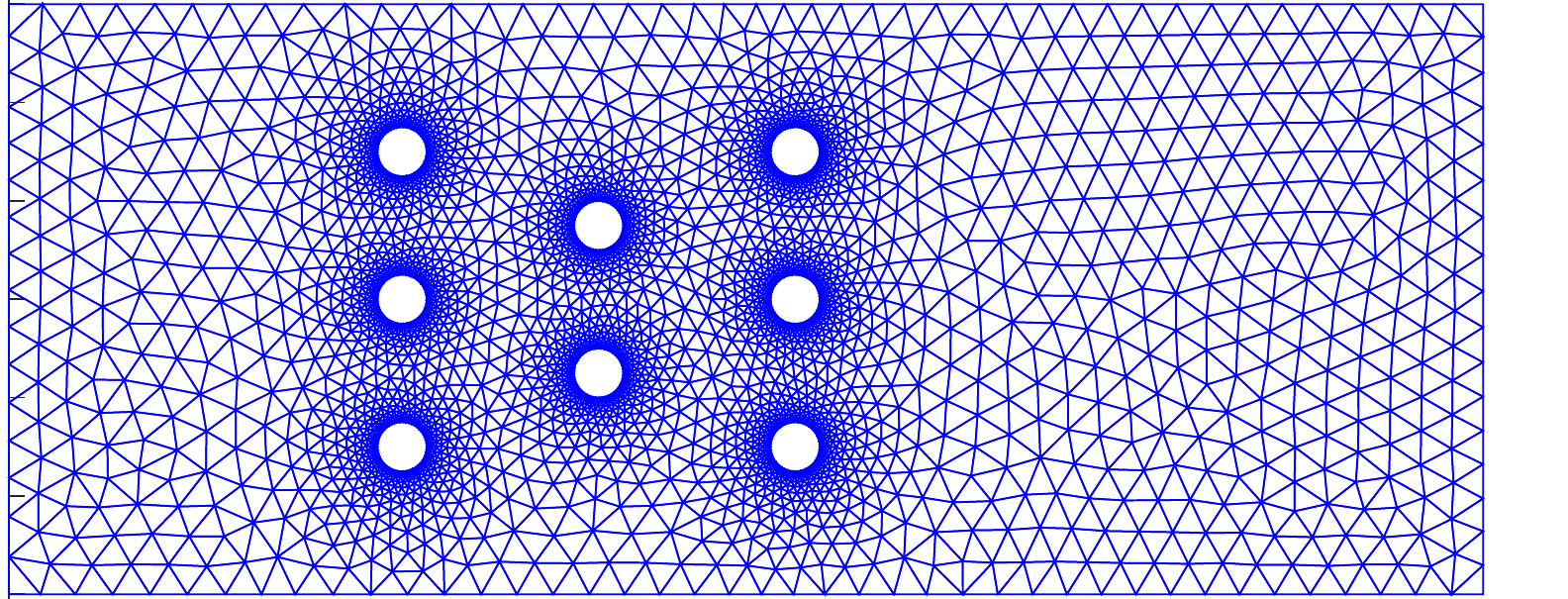}
\caption{Example II: Unstructured mesh for a rectangular box with eight holes.}
\label{fat_box}
\end{center}
\end{figure}
\vspace{-5mm}

\begin{figure}[h!t!b!]
\subfigure[Level 6]{
\includegraphics[trim=3.5cm 11.15cm 2cm 10.65cm,clip=true,width=0.3\columnwidth]{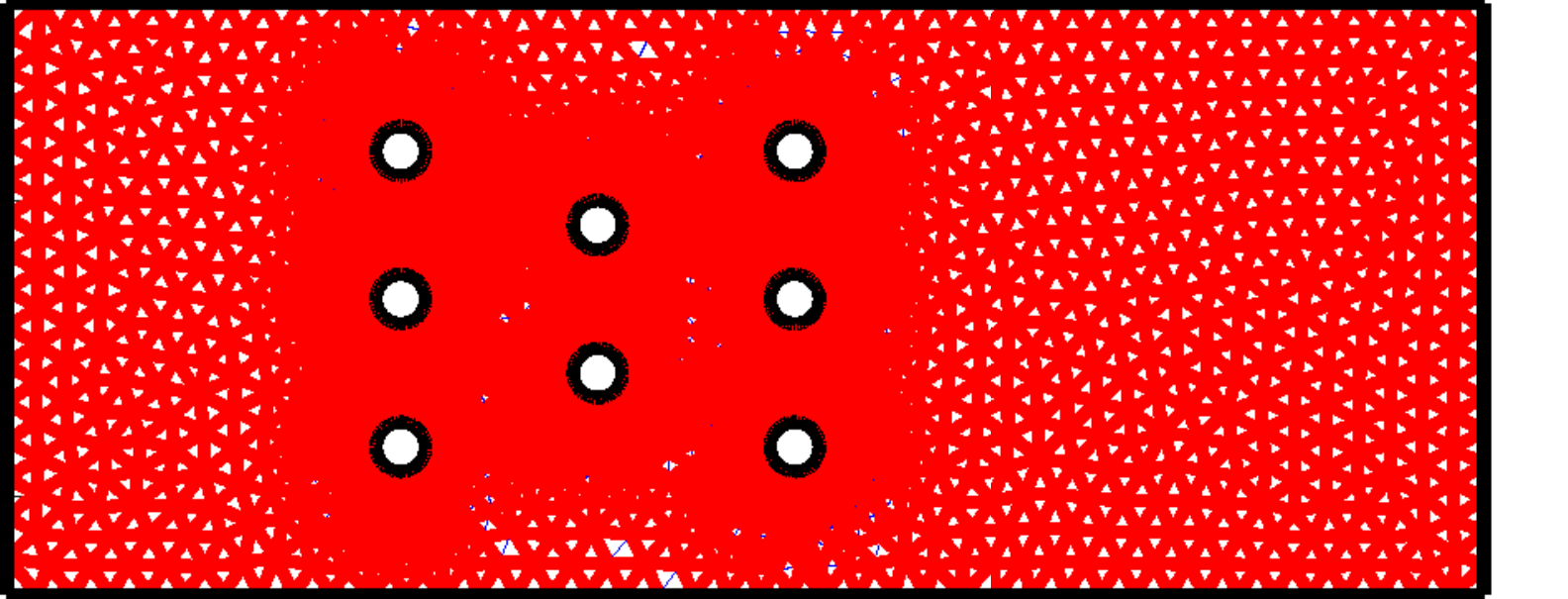}
}
\subfigure[Level 5]{
\includegraphics[trim=3.5cm 11.15cm 2cm 10.65cm,clip=true,width=0.3\columnwidth]{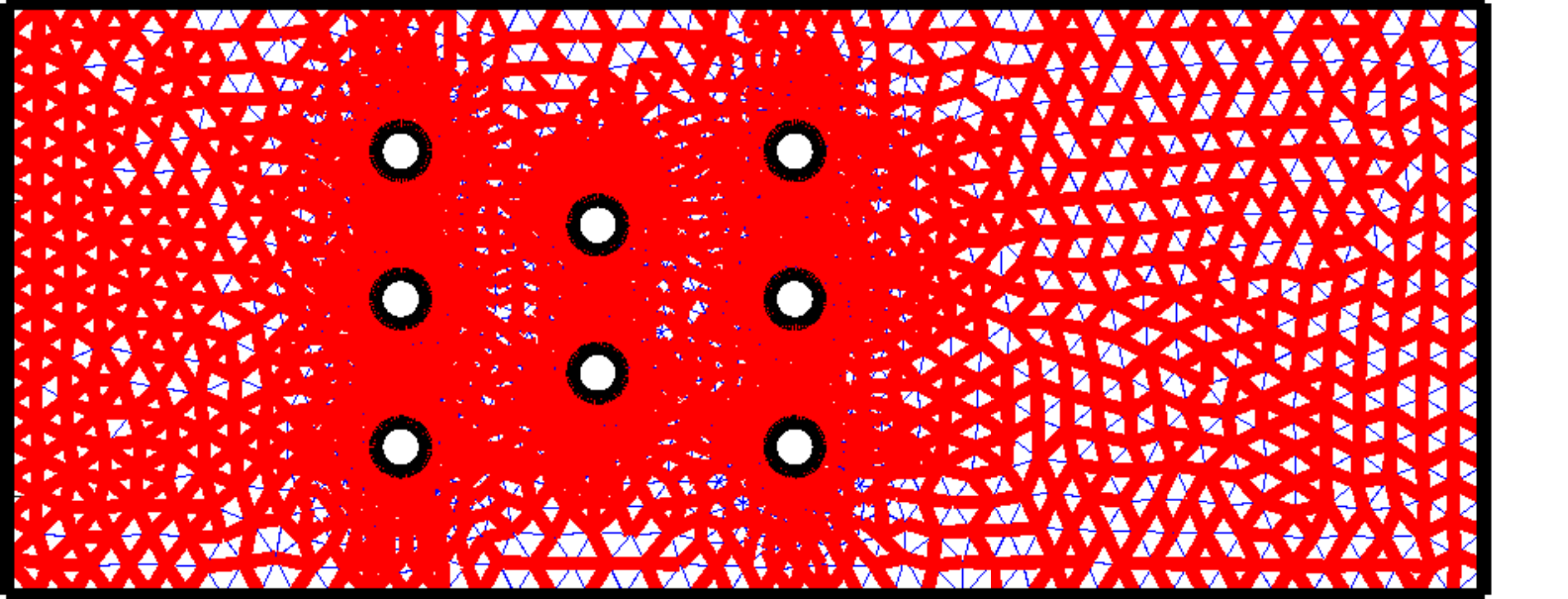}
}
\subfigure[Level 4]{
\includegraphics[trim=3.5cm 11.15cm 2cm 10.65cm,clip=true,width=0.3\columnwidth]{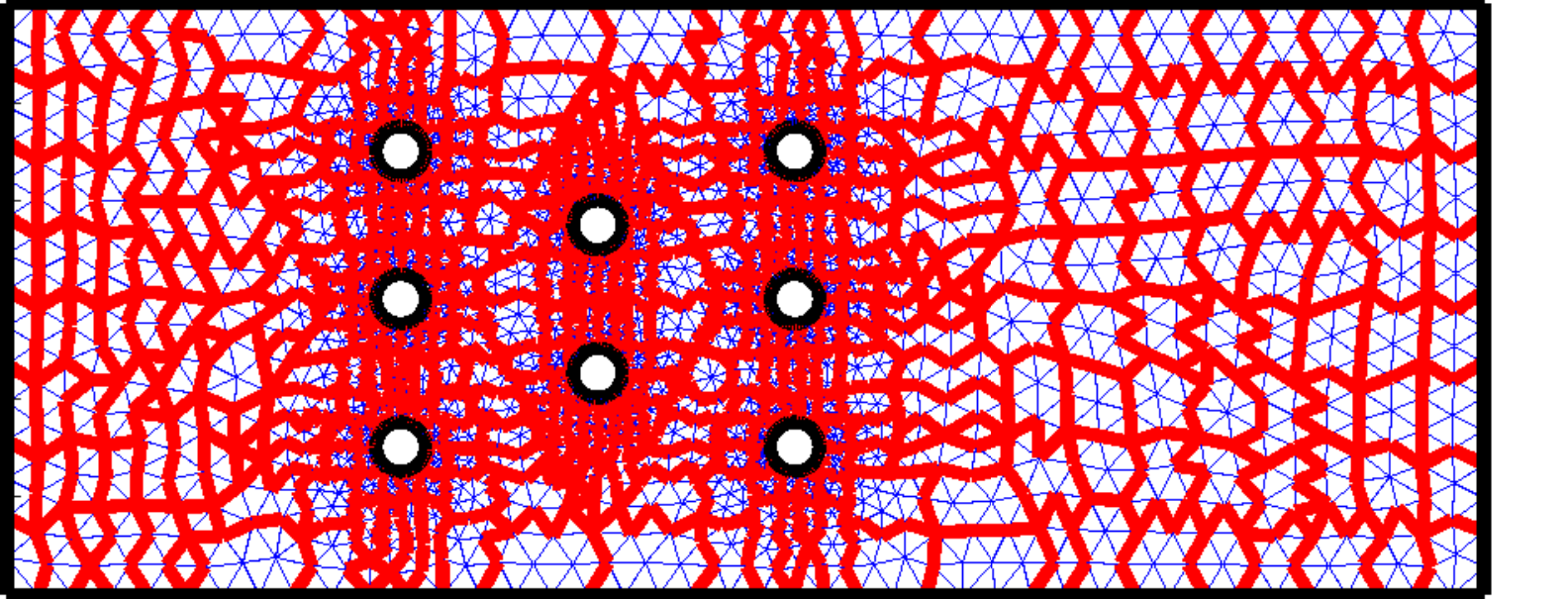}
}
\subfigure[Level 3]{
\includegraphics[trim=3.5cm 11.15cm 2cm 10.65cm,clip=true,width=0.3\columnwidth]{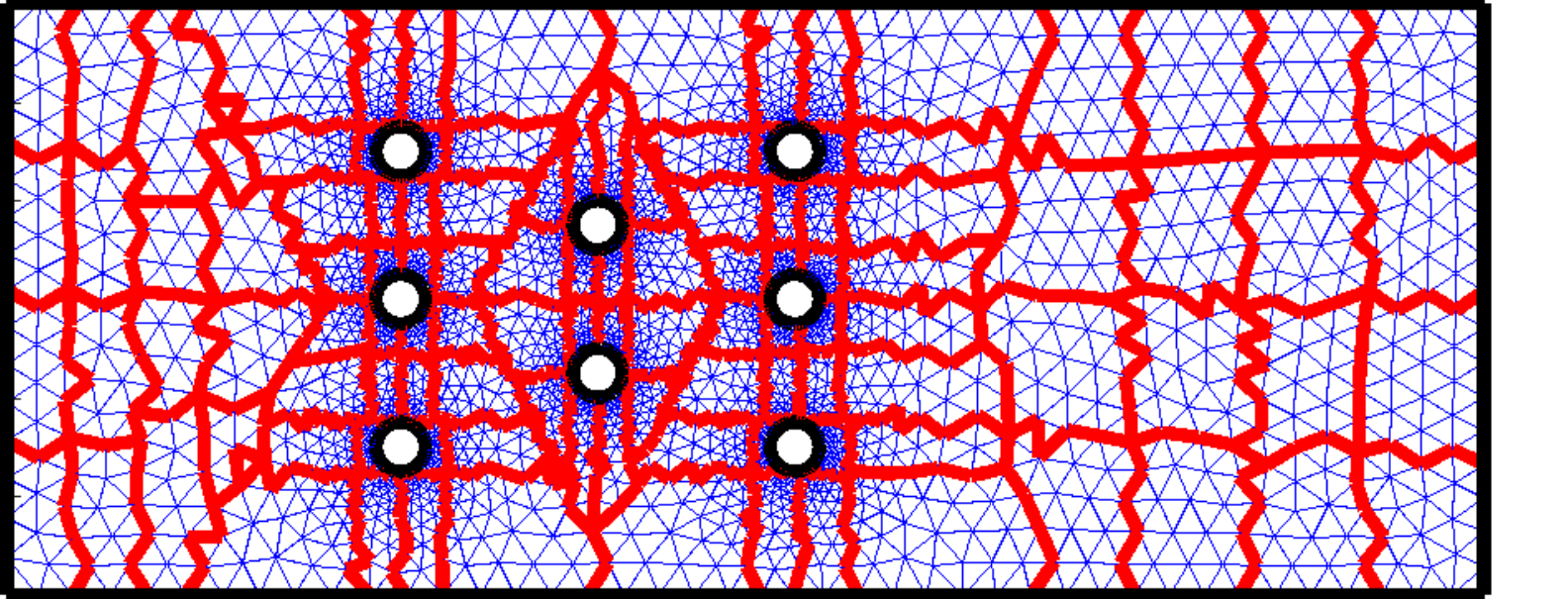}
}
\subfigure[Level 2]{
\includegraphics[trim=3.5cm 11.15cm 2cm 10.65cm,clip=true,width=0.31\columnwidth]{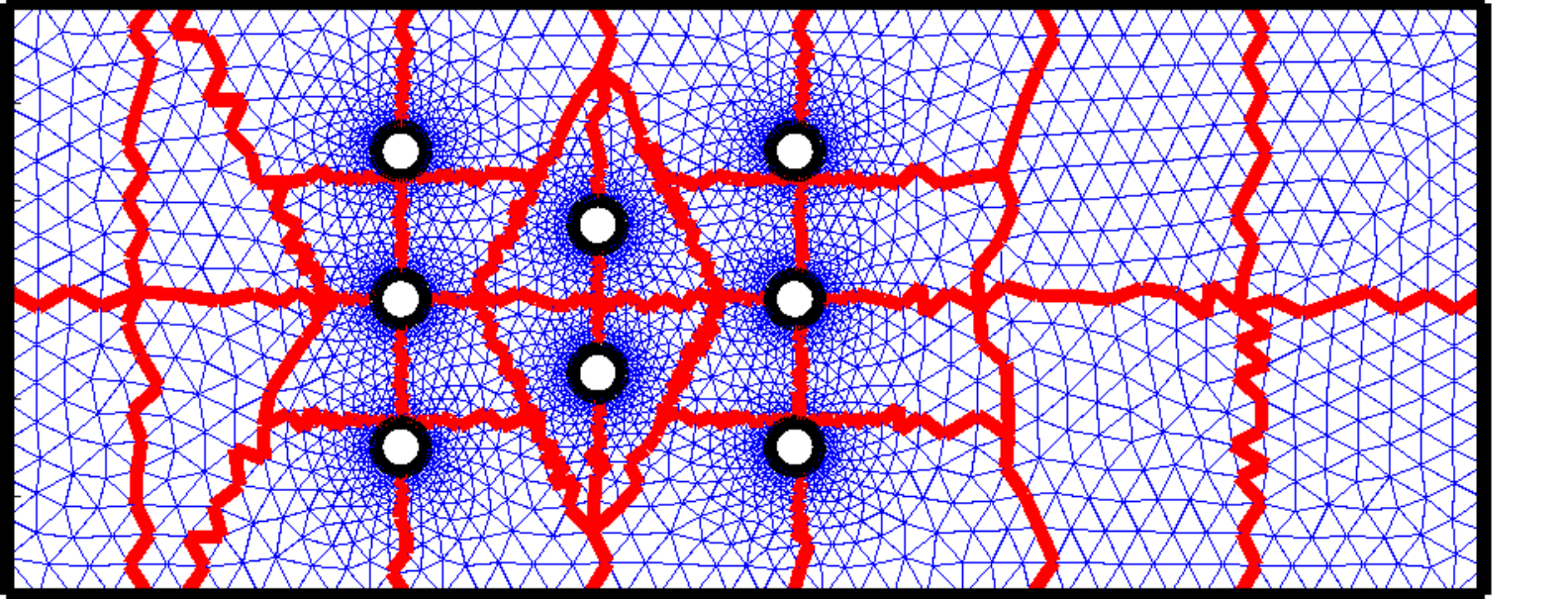}
}
\subfigure[Level 1]{
\includegraphics[trim=2.54cm 13.82cm 7.5cm 9.35cm,clip=true,width=0.29\columnwidth]{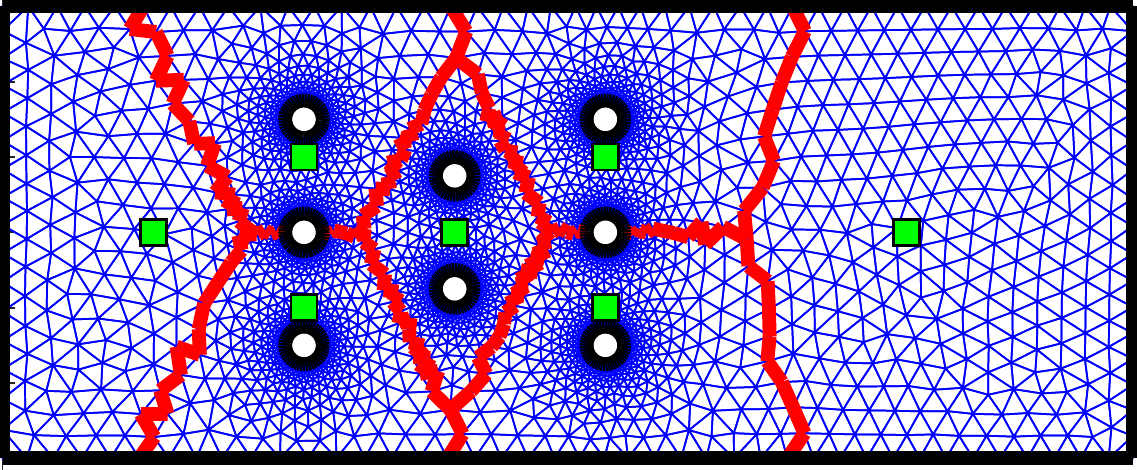}
}
\caption{Example II. Coarsening strategy 1: the seed points are marked with green squares in level $1$.}
\vspace{-7mm}
\figlab{cs1}
\end{figure}
\begin{figure}[h!t!b!]
\subfigure[Level 6]{
\includegraphics[trim=3.2cm 11cm 2cm 10.5cm,clip=true,width=0.3\columnwidth]{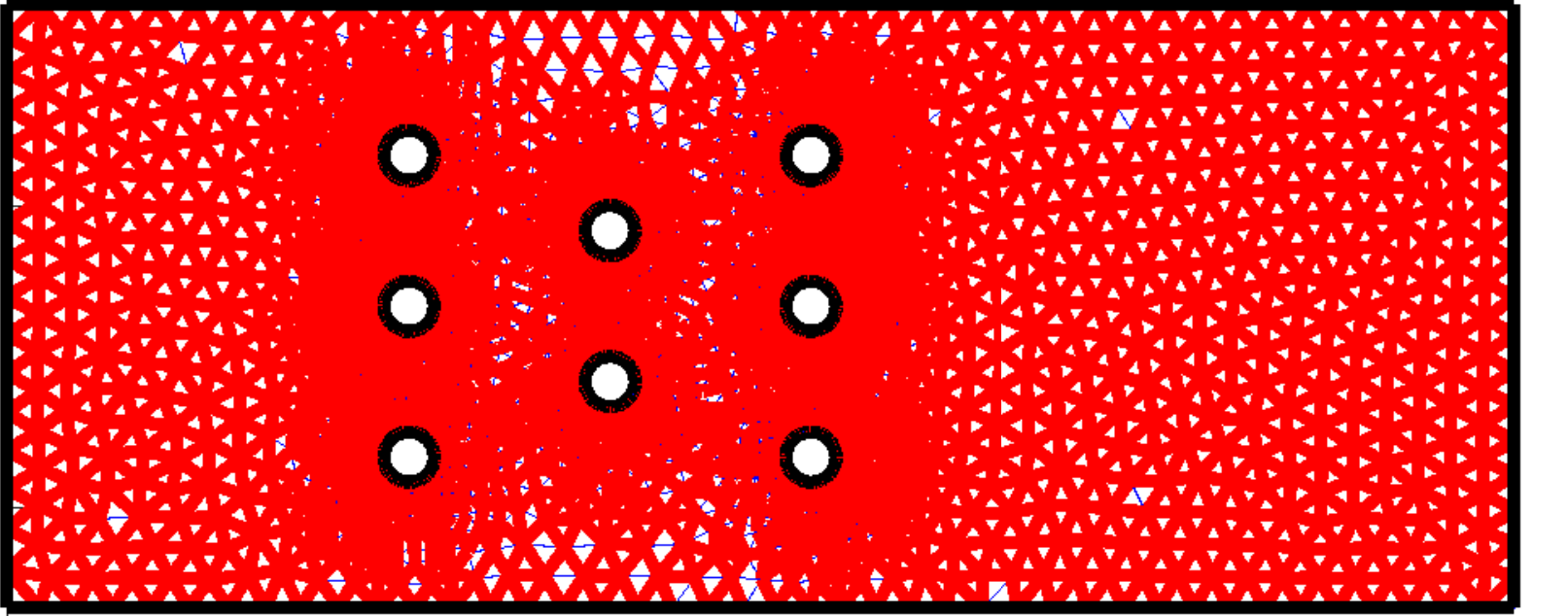}
}
\subfigure[Level 5]{
\includegraphics[trim=3.5cm 11.15cm 2cm 10.65cm,clip=true,width=0.3\columnwidth]{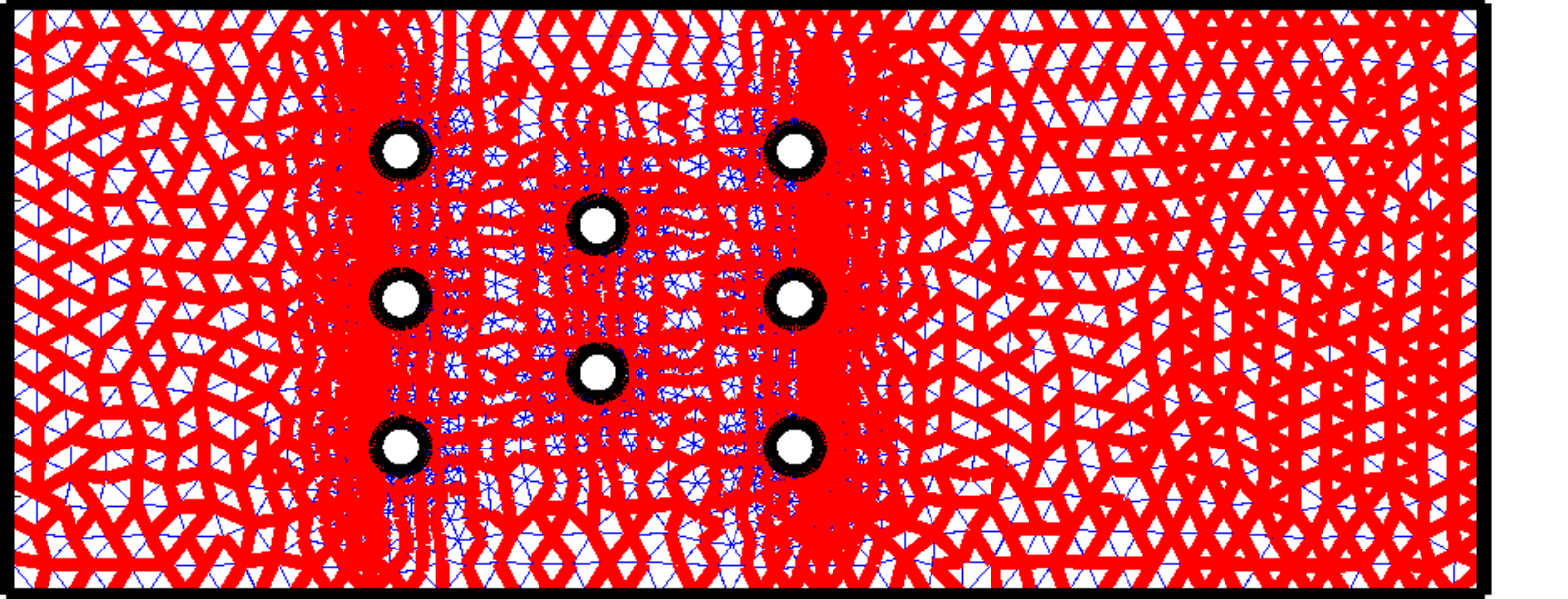}
}
\subfigure[Level 4]{
\includegraphics[trim=3.5cm 11.15cm 2cm 10.65cm,clip=true,width=0.3\columnwidth]{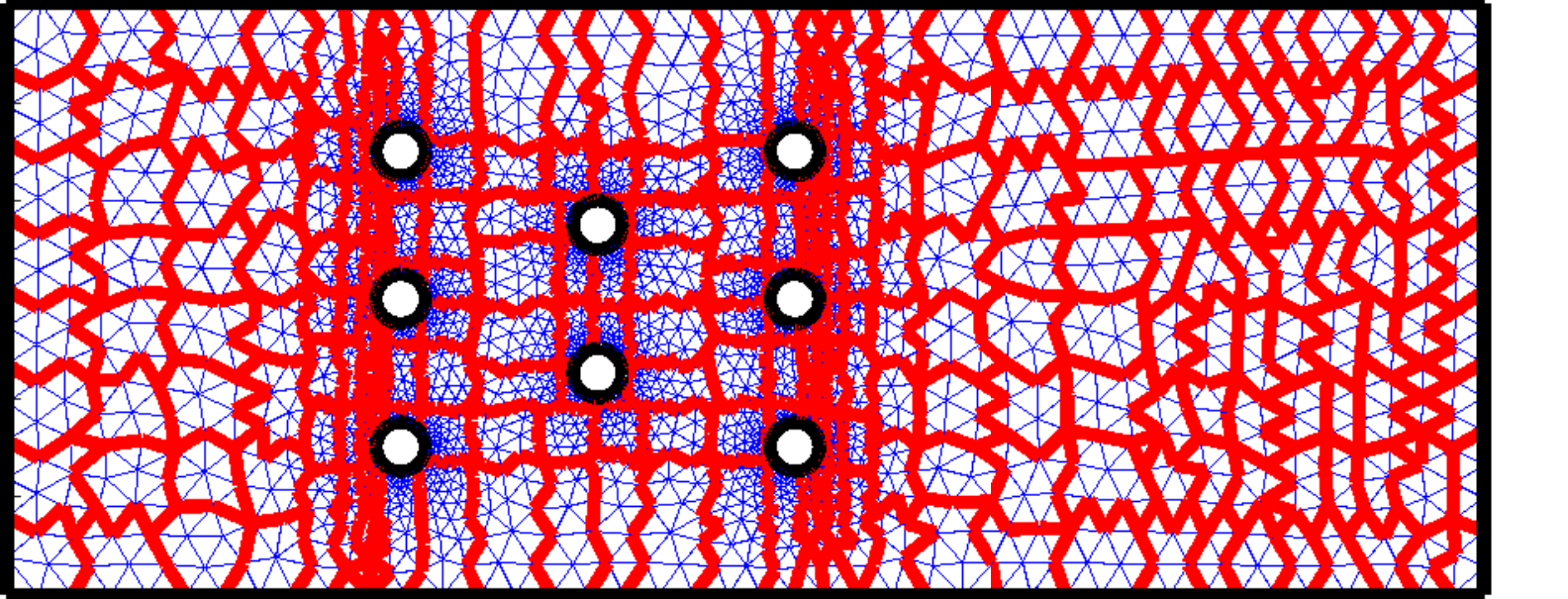}
}
\subfigure[Level 3]{
\includegraphics[trim=3.5cm 11.15cm 2cm 10.65cm,clip=true,width=0.3\columnwidth]{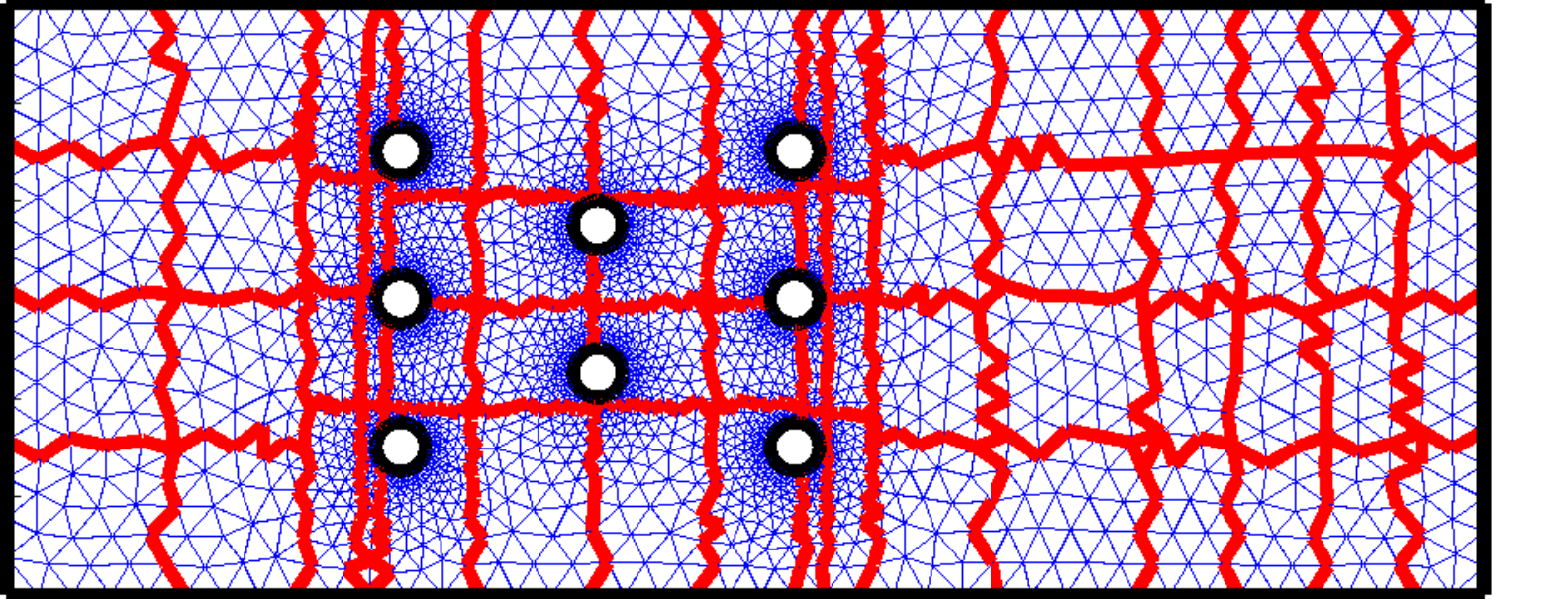}
}
\subfigure[Level 2]{
\includegraphics[trim=3.5cm 11.15cm 2cm 10.65cm,clip=true,width=0.31\columnwidth]{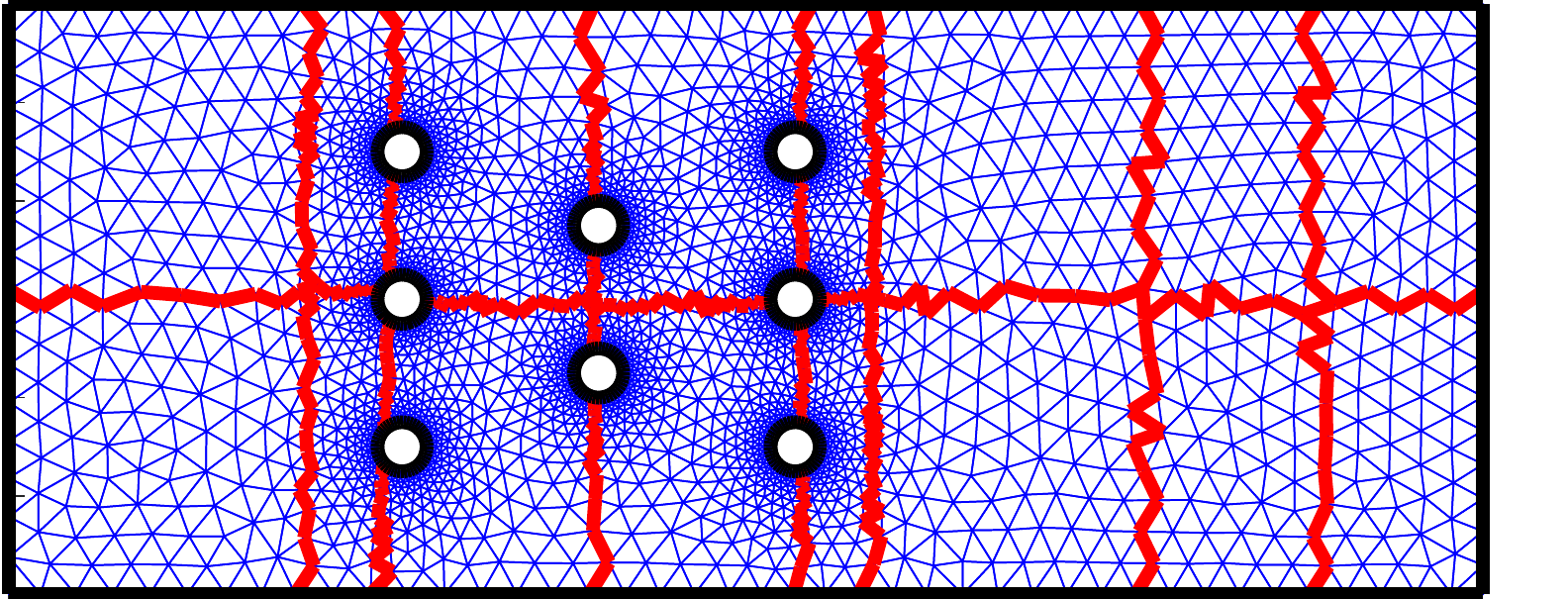}
}
\subfigure[Level 1]{
\includegraphics[trim=2.6cm 14cm 7.5cm 9.3cm,clip=true,width=0.29\columnwidth]{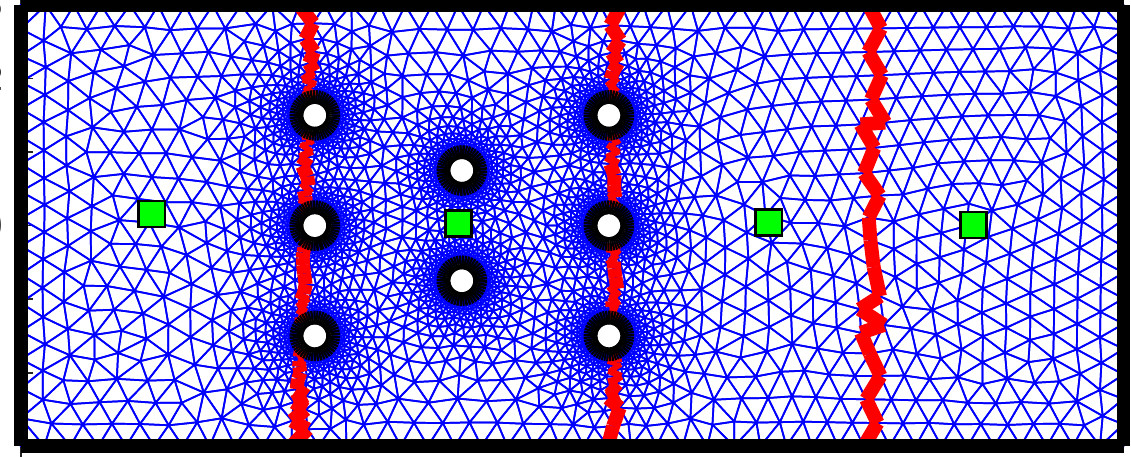}
}
\caption{Example II. Coarsening strategy 2: the seed points are marked with green squares in level $1$.}
\vspace{-7mm}
\figlab{cs2}
\end{figure}
Part of our future work
is to explore the coarsening strategies similar to the ones in
\cite{venkatakrishnan1993unstructured,venkatakrishnan1995agglomeration}.
This may give better coarsened levels with the number of macro-elements 
 reduced/increased by a constant factor between successive levels. The intergrid transfer
operators necessary for moving  residuals/errors between  levels is described next.

\subsection{Intergrid transfer operators}
\seclab{intergrid_transfer}
In standard nested multigrid algorithms, the intergrid transfer
operators are straightforward.  A popular approach is to take
injection for the prolongation and its adjoint for the
restriction. For skeletal system care must be taken in the
construction of these operators to ensure the convergence of the multigrid algorithm under consideration.
In the following we
propose ``physics-based energy-preserving'' operators using the fine scale DtN maps.

\subsubsection{Prolongation} The prolongation operator, $\I_k: M_{k-1}\longrightarrow M_{k}$, transfers the error
from the grid level $k-1$ to the finer grid level $k$.
Note that standard prolongation using injection would set the values on the interior edges to
 zero, and thus does not work.
For our multigrid algorithm it is defined as a function of the solution order $\p$ using  DtN maps. In particular:
\begin{itemize}
\item for $\p=1$,

    \begin{equation}
          \I_k:=
          \LRs{
              \begin{array}{c}
               -A_{k,II}^{-1} A_{k,IB} J_k \\
                  J_k
                \end{array}
            } \quad \text{for} \quad k=2,\cdots,N,
        \label{n_1_Ik}     
    \end{equation}
\item and for $\p>1$,

    \begin{equation}
          \I_k:=
          \begin{cases}
              J_N &\quad \text{for} \quad k=N,\\
          \LRs{
              \begin{array}{c}
               -A_{k,II}^{-1} A_{k,IB} J_k \\
                  J_k
                \end{array}
            } &\quad \text{for} \quad k=2,\cdots,N-1.
          \end{cases}
        \label{n_g1_Ik}     
    \end{equation}
\end{itemize}
Here, $J_k: M_{k-1}\longrightarrow M_{k}$ is the injection (interpolation). To understand the idea behind the  prolongation operator, let us consider $\p=1$ in \eqref{n_1_Ik}. 
First, through $J_k$, we interpolate the error
from the grid level $k-1$ to obtain error on the boundary edges of the finer
grid level $k$. 
Then we solve (via the first block in the definition of $\I_k$), for the error on the interior edges 
as a function of the interpolated error on the boundary. For $\p > 1$
the prolongation is defined in \eqref{n_g1_Ik}, namely, we use the same
prolongation as in the case of $\p = 1$ for all levels $k \le N -1$
and interpolate error from piecewise linear polynomials at level $N-1$
to piecewise $\p$th-order polynomials at level $N$.


\subsubsection{Restriction} The restriction operator, $\Q_{k-1}: M_{k}\longrightarrow M_{k-1}$, restricts the residual from level $k$ to
coarser level $k-1$. The popular idea is to define the restriction
operator as the adjoint (with respect to the $L^2$-inner products $\LRa{\cdot,\cdot}_{\E_k}$ and $\LRa{\cdot,\cdot}_{\E_{k-1}}$ in $M_k$ and $M_{k-1}$)  of the prolongation operator,
i.e. $\Q_{k-1}=\I_k^*$, and from our numerical studies (not shown
here) it still works well. However, this purely algebraic procedure,
though convenient, is not our desire.  Here, we construct the restriction
operator $\Q_{k-1}$ such that 
the coarse grid problem, via the Galerkin approximation, 
is exactly a discretized DtN problem on level $k-1$. 
To that end, let us define $\Q_{k-1}$ as:

\begin{itemize}
\item for $\p=1$,
\begin{equation}
    \Q_{k-1}:=
          \LRs{
              \begin{array}{c}
                  -J_k^* A_{k,BI} A_{k,II}^{-1} \quad 
                  J_k^*
                \end{array}
            } \quad \text{for} \quad k=2,\cdots,N,
          \label{n_1_Qk}
    \end{equation}
\item for $\p>1$,
\begin{equation}
    \Q_{k-1}:=
          \begin{cases}
              J_N^* &\quad \text{for} \quad k=N,\\
          \LRs{
              \begin{array}{c}
                  -J_k^* A_{k,BI} A_{k,II}^{-1} \quad
                  J_k^*
                \end{array}
            } &\quad \text{for} \quad k=2,\cdots,N-1.
          \end{cases}
          \label{n_g1_Qk}
    \end{equation}
\end{itemize}
Note that  if $A$ is
symmetric, then our defintion of the restriction operator $\Q_{k-1}$ is indeed the adjoint of the prolongation operator  $\I_{k}$. 

Given the prolongation and
restriction operators, we obtain our coarse grid equation using
the discrete Galerkin approximation \cite{trottenberg2000multigrid}
\begin{equation}
  \eqnlab{coarseEqn}
  \Q_{k-1}A_{k}\I_{k} \lambda_{k-1} = \Q_{k-1} g_k,
\end{equation}
where the coarse grid operator
\begin{equation}
    \label{coarse}
 A_{k-1} := \Q_{k-1}A_{k}\I_{k},
\end{equation}
for either $\p = 1$ and $k \le N$ or $\p > 1$ and $k \le N-1$, reads
\begin{multline}
 A_{k-1} = \LRs{
              \begin{array}{c}
                  -J_k^* A_{k,BI} A_{k,II}^{-1} \quad 
                  J_k^*
                \end{array}
 }
 \LRs{
    \begin{array}{cc}
      A_{k,II} & A_{k,IB} \\
      A_{k,BI} & A_{k,BB}
    \end{array}
  }
  \LRs{
    \begin{array}{c}
      -A_{k,II}^{-1} A_{k,IB} J_k \\
      J_k
    \end{array}
  }
  \\
  = J_k^*\LRp{A_{k,BB}-A_{k,BI}A_{k,II}^{-1}A_{k,IB}}J_k,
  \eqnlab{Akm1}
\end{multline}
and
\begin{equation}
  \eqnlab{ANm1}
  A_{N-1} = J_N^*A_NJ_N,
\end{equation}
for $\p > 1$ and $k = N$.

{\it Energy preservation:} In order for the multigrid algorithms
  to converge the intergrid operators should be constructed in such a
  way that the ``energy'' does not increase when transferring information
  between a level to a finer one \cite{bramble1993multigrid,Gopalakrishnan09aconvergent}.
Note that ``energy'' here need not necessarily  be associated with some physical energy. 
Indeed, if we associate $A_k$
 with a 
  bilinear form $a_k\LRp{.,.}$ such that $a_k\LRp{\kappa,\mu}= \LRa{A_k \kappa,\mu}_{\E_k}, \forall \kappa,\mu \in M_{k}$, where again $\LRa{.,.}_{\E_k}$ represents the L$^2$ inner product on $\E_k$, then we call $a_k\LRp{\kappa,\kappa}$ the energy on level $k$ associated with $\kappa$. Non-increasing energy means
\[
    a_k\LRp{\I_k\lambda,\I_k\lambda}\leq a_{k-1}\LRp{\lambda,\lambda} \quad \forall \lambda \in M_{k-1}, \quad \forall k=2,3,\cdots,N. 
\]
\begin{proposition}[Energy preservation]
  The proposed multigrid approach preserves the energy in the following sense: $\forall k=2,3,\cdots,N$,
  \begin{equation}
    \eqnlab{energyConser}
    a_k\LRp{\I_k\lambda,\I_k\lambda}=a_{k-1}\LRp{\lambda,\lambda} \quad \forall \lambda \in M_{k-1}. 
\end{equation}
  \end{proposition}
\begin{proof}
We proceed first with $\p=1$. From the definition of $A_k$ in
\eqref{trace_partition} and the definition of the prolongation operator $\I_k$ in \eqref{n_1_Ik} we
have
\begin{align*}
  &a_k\LRp{\I_k\lambda,\I_k\lambda}
   \\
    &=\LRa{\LRs{
    \begin{array}{cc}
      -A_{k,II}^{-1} A_{k,IB} J_k\lambda,
      J_k\lambda
    \end{array}
  },\LRs{
    \begin{array}{cc}
      A_{k,II} & A_{k,IB} \\
      A_{k,BI} & A_{k,BB}
    \end{array}
  }
  \LRs{
    \begin{array}{c}
      -A_{k,II}^{-1} A_{k,IB} J_k\lambda \\
      J_k\lambda
    \end{array}
  }}_{\E_k} \\
   &= \LRa{ \LRp{A_{k,BB}-A_{k,BI}A_{k,II}^{-1}A_{k,IB}}J_k \lambda, J_k \lambda}_{\E_k} \\
   &= \LRa{ J_k^*\LRp{A_{k,BB}-A_{k,BI}A_{k,II}^{-1}A_{k,IB}}J_k \lambda, \lambda}_{\E_{k-1}}
            = a_{k-1}\LRp{\lambda,\lambda},
\end{align*}
where the last equality comes from the definition of the coarse grid
operator $A_{k-1}$ in \eqnref{Akm1}.  For $\p>1$, we need to prove
\eqnref{energyConser} only for $k=N$, but this is straightforward since from \eqref{n_g1_Ik},
\eqref{n_g1_Qk}, and \eqnref{ANm1} we have
    \[
        a_N\LRp{\I_N\lambda,\I_N\lambda} = \LRa{A_N J_N\lambda,J_N\lambda}_{\E_{k}} = \LRa{J_N^* A_N J_N\lambda,\lambda}_{\E_{k-1}} = a_{k-1}\LRp{\lambda,\lambda}.
    \]
\end{proof}


  We now prove that the coarse grid operator \eqref{coarse} is also a
 discretized DtN map on every level.

  \begin{proposition}
    \propolab{DtN}
      At every level $k = 1,\hdots,N$, the Galerkin coarse grid operator \eqref{coarse} is a discretized DtN map on that level.
  \end{proposition}

  \begin{proof}
    The proof is a straightforward induction using the proposed coarsening strategy
    and the definition of the intergrid transfer operators.  First, consider
    the case of $\p=1$.   The fact that $A_N
    = A$ in \eqref{trace_linear_system} is a discretized DtN map on
    the finest level is clear by the definition of the hybridized
    methods.
 Assume at level $k$ the operator $A_k$ in \eqref{trace_partition} is
 a discretized DtN map. Taking $\lambda_{k,B} = J_k\lambda_{k-1}$ and
 condensing $\lambda_{k,I}$ out yield
 \[
 \LRp{A_{k,BB}-A_{k,BI}A_{k,II}^{-1}A_{k,IB}}J_k\lambda_{k-1} = g_{k,B}-A_{k,BI}A_{k,II}^{-1}g_{k,I}
 \]
 which, after restricting on the coarse space $M_{k-1}$ using $J^*_{k}$, becomes
 \[
 J_k^*\LRp{A_{k,BB}-A_{k,BI}A_{k,II}^{-1}A_{k,IB}}J_k\lambda_{k-1} = J^*_k\LRp{g_{k,B}-A_{k,BI}A_{k,II}^{-1}g_{k,I}},
 \]
 which is exactly the coarse grid equation \eqnref{coarseEqn}. That
 is, $A_{k-1}$ is the Schur complement obtained by eliminating all the trace unknowns inside all the macro-elements on level $k-1$. By definition, the coarse grid operator
 $A_{k-1}$ on level $\LRp{k-1}$ is also a discrete DtN map.

 For the case of $\p>1$, it is sufficient to show that $A_{N-1}$ is a
 discrete DtN map, but this is clear by: 1) taking $\lambda = J_N
 \lambda_{N-1}$ in \eqref{trace_linear_system}, 2) restricting both
 sides to $M_{N-1}$ using $J^*_N$, 3) recalling that $A_N$ is a discretized DtN
 map, and 4) observing that the resulting equation coincides with the
 coarse grid equation \eqnref{ANm1}.
  \end{proof}

\subsection{Smoothing}
\seclab{smoothing}
Let us denote the smoothing operator at level $k$ by $G_{k,m_k}$, where
$m_k$ stands for the number of smoothing steps performed at level
$k$. We take $m_k$ to satisfy
$$\beta_0m_k \leq m_{k-1} \leq \beta_1m_k,$$ with $\beta_1\ge \beta_0
> 1$. That is, the number of smoothing steps are increased as the mesh
is coarsened.  
The reason for this choice is based on the theoretical analysis for
non-nested multigrid methods in \cite{bpx1991}. However, as numerically shown
later in Section \secref{exp5}, even constant number of smoothing
steps at all levels, e.g. $2$, works well.

\subsection{Local correction operator}
\seclab{local_correction}
To motivate the need of local correction operator, let us consider the following decomposition of the skeletal space
$M_k = {\overline M}_k \oplus \hat{M}_k$ such that
     \[
         \lambda_k = {\overline\lambda}_k + \hat{\lambda}_k, \quad {\overline\lambda}_k \in {\overline M}_k \text{ and } \hat{\lambda}_k \in  \hat{M}_k,
     \]
where ${\overline\lambda}_k= {\overline\lambda}_{k,I}$ is given by the local correction $T_k: M_{k}\longrightarrow M_{k}$:
\begin{equation}
  \eqnlab{localCorrection}
         {\overline\lambda}_k := T_k          \LRs{
             \begin{array}{c}
                 g_{k,I} \\
                 0
             \end{array}
         } :=
         \LRs{
         \begin{array}{cc}
             A_{k,II}^{-1} & 0 \\
             0 & 0
         \end{array}
         }
         \LRs{
             \begin{array}{c}
                 g_{k,I} \\
                 0
             \end{array}
         },
     \end{equation}
     and the ``global component'' $\hat{\lambda}_k=\LRs{\hat{\lambda}_{k,I} \quad \hat{\lambda}_{k,B}}^T$ by
     \[
         \LRs{
      \begin{array}{cc}
        A_{k,II} & A_{k,IB} \\
          A_{k,BI} & A_{k,BB}
      \end{array}
       }
       \LRs{
       \begin{array}{c}
           \hat{\lambda}_{k,I} \\
           \hat{\lambda}_{k,B}
       \end{array}
       }
       =
       \LRs{
       \begin{array}{c}
           0 \\
           g_{k,B} - A_{k,BI}A_{k,II}^{-1}g_{k,I}
       \end{array}
       }.
   \]
Let us now consider $\p=1$. 
Standard two-grid cycles, for example, include: 1)
smoothing iterations on the system \eqref{trace_linear_system} for the
fine grid, 2) restricting the residual to the coarse grid, 3)
performing the coarse grid correction, and 4) prolongating the error
back to the fine grid. However, with the prolongation operator defined
in equation \eqref{n_1_Ik}, only $\hat{\lambda}_k$  can be recovered directly. In other
words, the burden of capturing $\overline{\lambda}_k$ is left for the smoother and/or
multigrid iterations.
Indeed, in our numerical studies we found that this standard algorithm takes more multigrid iterations and still
converges with Gauss-Seidel and Chebyshev accelerated Jacobi smoothers.
However, it diverges with block-Jacobi smoother at low orders as shown in Table \ref{tab:HDG_tau1_bj_without_local}.

This issue can be fixed using the concept of subspace 
correction \cite{xu1992iterative,trottenberg2000multigrid}. We can perform a local correction
in the subspace ${\overline M}_k$, as in \eqnref{localCorrection}, either before or after the coarse-grid correction (or both if 
symmetry is desirable). 
With this local correction, we have observed that even with
Jacobi-type smoothers the multigrid algorithm converges.
 It is important to point out that the application of the local correction
operator $T_k$ is completely local to each macro-element at the $\LRp{k-1}$th level, and
thus can be carried out in parallel.


\subsection{Relationship to AMG operators}

The intergrid transfer operators in section
\secref{intergrid_transfer} and the local correction operator in
\secref{local_correction} are closely related to the ideal
interpolation and smoothing operators in algebraic multigrid (AMG)
literature \cite{trottenberg2000multigrid}. However, the ideal
operators in AMG lead to a direct solver strategy, namely nested
dissection \cite{george1973nested} and suffer from large memory
requirements. This renders the algorithm impractical for large scale
simulations especially in 3D. In that respect, the coarsening in our
operators is the key as it leads to an $\mc{O}(N)$ iterative algorithm
and can be applied to large scale problems. In fact, this approach is
pursued in the context of finite differences in
\cite{reusken1994multigrid,reusken1993multigrid,dendy1982black,de1990matrix}
under the name of Schur complement multigrid methods.  However, our
approach is more general in the sense that it can be applied to any
unstructured mesh, whereas all the previous works deal with
structured meshes. As shown in
Figure \ref{fat_box}, it does not require the meshes to be nested. In fact, our approach can be considered as a combination
of AMG and geometric multigrid methods. As such it benefits from the
robustness of AMG and the fixed coarse grid construction costs of
geometric multigrid.

\subsection{Multigrid V-cycle algorithm}
We are now in the position to define our multigrid algorithm. 
We begin with a fixed point scheme on the finest level:
$$\lambda^{i+1} = \lambda^i + B_N(r_N), \quad i = 0,\hdots$$
where $r_N = g_N - A_N\lambda^i$ and  the action of $B_N$ on a function/vector is defined recursively using the multigrid algorithm \ref{al:multigrid}. 

\begin{algorithm}
  \begin{algorithmic}[1]
 \STATE {\it Initialization}: \\
	$e^{\{0\}} = 0,$ 
 \STATE {\it Presmoothing}: \\ 
	$e^{\{1\}} = e^{\{0\}} + G_{k,m_k}\left(r_k-A_ke^{\{0\}}\right),$
\STATE {\it Local Correction}: \\ 
	$e^{\{2\}} = e^{\{1\}} + T_k\left(r_k-A_ke^{\{1\}}\right),$
 \STATE {\it Coarse Grid Correction}: \\ 
	$e^{\{3\}} = e^{\{2\}} + I_k B_{k-1} \LRp{Q_{k-1}\left(r_k-A_ke^{\{2\}}\right)},$
\STATE {\it Local Correction}: \\ 
	$e^{\{4\}} = e^{\{3\}} + T_k\left(r_k-A_ke^{\{3\}}\right),$
 \STATE {\it Postsmoothing}: \\ 
     $B_k\LRp{r_k} = e^{\{5\}} = e^{\{4\}} + G_{k,m_k}\left(r_k-A_ke^{\{4\}}\right).$
\end{algorithmic}
  \caption{Multigrid v-cycle algorithm}
  \label{al:multigrid}
\end{algorithm}
At the coarsest level $M_1$, we set $B_1 = A_1^{-1}$ and the inversion
is computed using direct solver.  Note that both local correction
steps 3 and 5 are presented for the sake of symmetry of the
algorithm. In practice, we use either step 3 or step 5.  For example,
numerical results in Section~\secref{numerical} use only step 3.  
\section{Numerical results}
\seclab{numerical}
In this section, we study the performance of multigrid both as a
solver and as a left preconditioner. Since the global trace system of
NIPG-H and IIPG-H methods are not symmetric, we use preconditioned
GMRES for all the methods to enable a direct comparison.  For all the numerical examples:
1) the stopping
tolerance is taken as $10^{-9}$ for
the normalized residual (normalized by the norm of the right hand side
of \eqref{trace_linear_system}), and 2) following \cite{bpx1991} we choose the smoothing parameters $\beta_0$ and $\beta_1$ (see Section \secref{smoothing}) as $2$.
In all the tables, $``*"$ stands for
either the divergence of the GMRES/multigrid solver or $200$  iterations was already taken
but the normalized residual was still larger than the tolerance. 

In examples 1-5 in Sections \secref{exp1}-\secref{exp5}, we consider HDG and hybridized
IPDG methods. In example 6 in Section \secref{multinumerics}, we consider the case of multnumerics with mixed methods (RT-H)
in some parts of the domain and NIPG-H in other parts. Though we have also tested the multigrid algorithm for hybridized RT1 and RT0
methods for example I and obtained scalable results, they are not shown here due to the page limitation.

\subsection{Example I: Poisson equation in a unit square}
\seclab{exp1}
In this first example, we consider the Poisson equation in a unit
square. We take the exact solution
to be of the form $\pres=x y e^{x^2y^3}$.
Dirichlet boundary condition using the 
exact solution is applied directly (strongly) via the trace unknowns. We consider  hybridized DG methods and study the
performance of multigrid both as a solver and as a preconditioner.
  The number of levels in the multigrid hierarchy and the corresponding number of quadrilateral elements 
    are shown in Table \ref{tab:level_elements_square}.
\begin{table}[h!b!t!]
\begin{center}
\begin{tabular}{ | r || c c  c  c  c  c | }
\hline
Levels &  2 & 3 & 4 & 5 & 6 & 7 \\
\hline
Elements & 16 & 64 & 256 & 1024 & 4096 & 16384 \\
\hline
\end{tabular}
\caption{\label{tab:level_elements_square} Example I: The multigrid hierarchy.} 
\end{center}
\end{table}

We consider the following four different smoothers and compare their
performance for each of the hybridized DG methods.
\begin{itemize}
    \item{Damped point-Jacobi with the relaxation parameter $\omega=2/3$ \cite{sundar2015comparison}.}
    \item{Chebyshev accelerated point-Jacobi method. This smoother requires the estimates of extreme eigenvalues and 
      we compute an approximate maximum eigenvalue $\Lambda_{max}$ with $\mc{O}\LRp{10^{-2}}$-accuracy.
Following \cite{adams2003parallel}, we estimate the smallest
eigenvalue using $\Lambda_{max}/30$. 
    }
    \item{Lower-upper symmetric point Gauss-Seidel method with one forward and one backward sweep
         during each iteration, and  we name this smoother as LU-SGS for simplicity.}
    \item{Block-Jacobi smoother where one block consists of all the
      degrees of freedom corresponding to an edge $\e \in \E_k$ at the $k$th level.
      Here,
      we do not use any damping as it does not make any difference in
      the number of iterations in our numerical studies.}
\end{itemize}

For the stabilization parameters, we consider a mesh-independent value
$\tau=1$ and mesh-dependent form $\tau=1/h_{min}$ for HDG. The former
corresponds to the upwind HDG proposed in
\cite{Bui-Thanh15,Bui-Thanh15a}.  
Following \cite{riviere2001priori},
we take the stabilization to be $\tau=(\p+1)(\p+2)/h_{min}$ for the
hybridized IPDG schemes. In later
examples, when the permeability ${\bf K}$ is spatially varying we also consider
$\tau$ as a function of ${\bf K}$.
We refer to \cite{MR2629996,Bui-Thanh15,cockburn2008superconvergent} for the $h$-convergence order of HDG
schemes and \cite{riviere2001priori} for the IPDG schemes.

\begin{table}[h!b!t!]
\centering
\begin{tabular}{|r|c|c|c|c|c|c||c|c|c|c|c|c|}
\hline
& \multicolumn{6}{c||}{MG as solver} & \multicolumn{6}{c|}{MG
with GMRES}\\
\cline{2-13}
\!\!\! $\p$ \!\!\!\! &  \multicolumn{6}{c||}{\!\!\scriptsize  Levels\!\!} &  \multicolumn{6}{c|}{\!\!\scriptsize Levels\!\!}\\
\hline
& \scriptsize 2 & \scriptsize 3 & \scriptsize 4 & \scriptsize 5 & \scriptsize 6 & \scriptsize 7 & \scriptsize 2 & \scriptsize 3 & \scriptsize 4 & \scriptsize 5 & \scriptsize 6 & \scriptsize 7 \\
\cline{2-13}
1 & 14 & 14 & 14 & 15 & 15 & 15 & 6 & 8 & 8 & 8 & 8 & 8\\
2 & 14 & 14 & 15 & 15 & 15 & 15 & 8 & 8 & 8 & 8 & 8 & 8\\
3 & 15 & 16 & 16 & 16 & 16 & 16 & 8 & 9 & 9 & 9 & 9 & 8\\
4 & 19 & 19 & 19 & 19 & 19 & 19 & 9 & 10 & 10 & 9 & 9 & 9\\
5 & 21 & 21 & 21 & 21 & 21 & 21 & 10 & 10 & 10 & 10 & 10 & 9\\
6 & 23 & 23 & 24 & 24 & 24 & 24 & 10 & 11 & 11 & 11 & 10 & 10\\
7 & 25 & 25 & 25 & 25 & 25 & 25 & 11 & 11 & 11 & 11 & 11 & 11\\
8 & 27 & 27 & 28 & 28 & 28 & 28  & 11 & 12 & 11 & 11 & 11 & 11\\
9 & 30 & 30 & 30 & 29 & 29 & 29  &12  &12  &12  &12  &12  &11 \\
10 &30  &31  &31  &31  &31  &32   &12  &12  &12  &12  &12  &12 \\
\hline
\end{tabular}
\caption{\label{tab:HDG_tau1_J} Example I. HDG with stabilization
  $\tau=1/h_{min}$: the number of iterations for multigrid with point-Jacobi 
    smoother as solver and preconditioner.}
\end{table}

\begin{table}[h!b!t!]
\centering
\begin{tabular}{|r|c|c|c|c|c|c||c|c|c|c|c|c|}
\hline
& \multicolumn{6}{c||}{MG as solver} & \multicolumn{6}{c|}{MG
with GMRES}\\
\cline{2-13}
\!\!\! $\p$ \!\!\!\! &  \multicolumn{6}{c||}{\!\!\scriptsize  Levels\!\!} &  \multicolumn{6}{c|}{\!\!\scriptsize Levels\!\!}\\
\hline
& \scriptsize 2 & \scriptsize 3 & \scriptsize 4 & \scriptsize 5 & \scriptsize 6 & \scriptsize 7 & \scriptsize 2 & \scriptsize 3 & \scriptsize 4 & \scriptsize 5 & \scriptsize 6 & \scriptsize 7 \\
\cline{2-13}
1 & 13 & 14 & 14 & 14 & 14 & 14 & 6 & 8 & 8 & 8 & 8 & 8\\
2 & 9 & 10 & 10 & 10 & 10 & 10 & 6 & 7 & 7 & 7 & 7 & 7\\
3 & 27 & 29 & 30 & 31 & 31 & 31 & 12 & 12 & 12 & 12 & 12 & 12\\
4 & 31 & 29 & 29 & 29 & 28 & 28 & 12 & 12 & 12 & 11 & 11 & 11\\
5 & 31 & 32 & 32 & 32 & 32 & 32 & 13 & 13 & 12 & 12 & 12 & 12\\
6 & 34 & 33 & 32 & 32 & 32 & 31 & 13 & 13 & 12 & 12 & 12 & 11\\
7 & 32 & 32 & 32 & 32 & 32 & 32 & 13 & 12 & 12 & 12 & 12 & 12\\
8 & 33 & 33 & 33 & 33 & 33 & 33 & 13 & 12 & 12 & 12 & 12 & 12\\
9 & 33 & 33 & 33 & 33 & 33 & 33 & 13 & 13 & 13 & 12 & 12 & 12\\
10 & 33 & 33 & 33 & 33 & 33  & 33 &13  &13  &13  &12  &12  &12 \\
\hline
\end{tabular}
\caption{\label{tab:HDG_tau1_cJ} Example I. HDG with stabilization $\tau=1/h_{min}$: the number of iterations for multigrid with Chebyshev accelerated point-Jacobi smoother as solver and preconditioner.}
    \vspace{-2mm}
\end{table}

\begin{table}[h!b!t!]
\centering
\begin{tabular}{|r|c|c|c|c|c|c||c|c|c|c|c|c|}
\hline
& \multicolumn{6}{c||}{MG as solver} & \multicolumn{6}{c|}{MG
with GMRES}\\
\cline{2-13}
\!\!\! $\p$ \!\!\!\! &  \multicolumn{6}{c||}{\!\!\scriptsize  Levels\!\!} &  \multicolumn{6}{c|}{\!\!\scriptsize Levels\!\!}\\
\hline
& \scriptsize 2 & \scriptsize 3 & \scriptsize 4 & \scriptsize 5 & \scriptsize 6 & \scriptsize 7 & \scriptsize 2 & \scriptsize 3 & \scriptsize 4 & \scriptsize 5 & \scriptsize 6 & \scriptsize 7 \\
\cline{2-13}
1 & 5 & 5 & 5 & 5 & 5 & 5 & 4 & 4 & 4 & 5 & 5 & 5\\
2 & 5 & 5 & 5 & 5 & 5 & 5 & 4 & 4 & 4 & 4 & 4 & 4\\
3 & 5 & 6 & 6 & 6 & 6 & 6 & 5 & 5 & 5 & 5 & 5 & 5\\
4 & 6 & 6 & 6 & 6 & 6 & 6 & 5 & 5 & 5 & 5 & 5 & 5\\
5 & 7 & 7 & 7 & 7 & 7 & 7 & 5 & 5 & 6 & 6 & 6 & 5\\
6 & 7 & 7 & 8 & 8 & 8 & 8 & 5 & 6 & 6 & 6 & 6 & 6\\
7 & 8 & 8 & 8 & 8 & 8 & 8 & 6 & 6 & 6 & 6 & 6 & 6\\
8 & 8 & 8 & 9 & 9 & 9 & 9 & 6 & 6 & 6 & 6 & 6 & 6\\
9 & 9 & 9 & 9 & 9 & 9 & 9  &7  &7  &7  &7  &7  &7 \\
10 &9  &9  &10  &10  &10  &10   &7  &7  &7  &7  &7  &7 \\
\hline
\end{tabular}
\caption{\label{tab:HDG_tau1_gs} Example I. HDG with stabilization $\tau=1/h_{min}$: the number of iterations for multigrid with LU-SGS smoother as solver and preconditioner.}
\end{table}

\begin{table}[h!b!t!]
\centering
\begin{tabular}{|r|c|c|c|c|c|c||c|c|c|c|c|c|}
\hline
& \multicolumn{6}{c||}{MG as solver} & \multicolumn{6}{c|}{MG
with GMRES}\\
\cline{2-13}
\!\!\! $\p$ \!\!\!\! &  \multicolumn{6}{c||}{\!\!\scriptsize  Levels\!\!} &  \multicolumn{6}{c|}{\!\!\scriptsize Levels\!\!}\\
\hline
& \scriptsize 2 & \scriptsize 3 & \scriptsize 4 & \scriptsize 5 & \scriptsize 6 & \scriptsize 7 & \scriptsize 2 & \scriptsize 3 & \scriptsize 4 & \scriptsize 5 & \scriptsize 6 & \scriptsize 7 \\
\cline{2-13}
1 & 7 & 7 & 8 & 8 & 8 & 8 & 4 & 5 & 6 & 6 & 6 & 6\\
2 & 6 & 7 & 8 & 8 & 9 & 9 & 4 & 5 & 6 & 6 & 6 & 6\\
3 & 8 & 9 & 9 & 9 & 9 & 9 & 6 & 6 & 6 & 6 & 6 & 6\\
4 & 9 & 10 & 10 & 10 & 10 & 10 & 6 & 7 & 7 & 7 & 7 & 7\\
5 & 11 & 12 & 12 & 12 & 12 & 12 & 6 & 8 & 8 & 8 & 8 & 7\\
6 & 12 & 12 & 13 & 13 & 13 & 13 & 7 & 8 & 8 & 8 & 8 & 8\\
7 & 13 & 14 & 14 & 14 & 14 & 15 & 7 & 8 & 8 & 8 & 8 & 8\\
8 & 14 & 15 & 15 & 15 & 15 & 15 & 8 & 9 & 9 & 9 & 9 & 8\\
9 & 15 & 16 & 16 & 16 & 17 & 17  & 8 & 9 & 9 & 9 & 9 & 9\\
10 &16  &17  &17  &17  &17  &17   &8  &9  &9  &9  &9  & 9\\
\hline
\end{tabular}
\caption{\label{tab:HDG_tau1_bj} Example I. HDG with stabilization $\tau=1/h_{min}$: the number of iterations for multigrid with block-Jacobi smoother as solver and preconditioner.}
    \vspace{-5mm}
\end{table}


%
%
In Tables \ref{tab:HDG_tau1_J}-\ref{tab:HDG_tau1_bj}, we study the
performance of multigrid as a solver and as a preconditioner for HDG
with stabilization $\tau=1/h_{min}$. 
We now present a few important observations from these tables. The
stabilization $\tau=1/h_{min}$ gives $h$-scalable results with all the
smoothers, i.e. the number of required multigrid/GMRES iterations is
almost unchanged when the mesh is refined. However, the actual number of
iterations depends on the effectiveness of the smoother under
consideration. As can be seen in Table \ref{tab:HDG_tau1_J} with
point-Jacobi smoother, the number of iterations increases with
high-order solutions when multigrid is used as a solver. In contrast, the
multigrid preconditioner is effective in keeping the number of
GMRES iterations almost unchanged (the increase is neligible). The results in Table
\ref{tab:HDG_tau1_cJ} shows that Chebyshev acceleration for
point-Jacobi smoother behaves the same as the point-Jacobi smoother in terms of scalability.
However, in terms of the number of iterations 
it takes either more (for multigrid as solver) or
about the same (for multigrid preconditioned GMRES).
Table \ref{tab:HDG_tau1_gs} shows that LU-SGS is the most effective
smoother, requiring the least number of multigrid/GMRES
iterations. Moreover, the number of both multigrid and GMRES iterations
are almost constant as either the mesh is refined or the solution order increases.
  However, it should be pointed out that the LU-SGS smoother requires twice the amount of work compared to
  point-Jacobi for each iteration due to one forward and one
  backward sweep. 
Table \ref{tab:HDG_tau1_bj} shows that
block-Jacobi is the second best smoother in terms of the number of
iterations; otherwise its scalability behavior is similar to the point-Jacobi smoother as the mesh or the solution order is refined.

\begin{table}[h!b!t!]
\centering
\begin{tabular}{|r|c|c|c|c|c|c||c|c|c|c|c|c|}
\hline
& \multicolumn{6}{c||}{MG as solver} & \multicolumn{6}{c|}{MG
with GMRES}\\
\cline{2-13}
\!\!\! $\p$ \!\!\!\! &  \multicolumn{6}{c||}{\!\!\scriptsize  Levels\!\!} &  \multicolumn{6}{c|}{\!\!\scriptsize Levels\!\!}\\
\hline
& \scriptsize 2 & \scriptsize 3 & \scriptsize 4 & \scriptsize 5 & \scriptsize 6 & \scriptsize 7 & \scriptsize 2 & \scriptsize 3 & \scriptsize 4 & \scriptsize 5 & \scriptsize 6 & \scriptsize 7 \\
\cline{2-13}
1 & 6 & 11 & 21 & 40 & 76 & 145 & 4 & 7 & 11 & 15 & 21 & 28\\
2 & 8 & 15 & 28 & 53 & 101 & 193 & 6 & 8 & 12 & 17 & 23 & 31\\
3 & 10 & 18 & 32 & 60 & 113 & * & 7 & 9 & 13 & 18 & 24 & 32\\
4 & 11 & 19 & 34 & 62 & 117 & * & 7 & 10 & 13 & 18 & 24 & 32\\
5 & 13 & 21 & 36 & 66 & 124 & * & 8 & 10 & 14 & 18 & 24 & 32\\
6 & 13 & 22 & 37 & 68 & 126 & * & 7 & 10 & 14 & 18 & 24 & 32\\
7 & 15 & 23 & 39 & 71 & 131 & * & 8 & 11 & 14 & 18 & 24 & 32\\
8 & 15 & 24 & 40 & 72 & 132 & * & 8 & 11 & 14 & 18 & 24 & 31\\
9 & 16 & 25 & 42 & 74 & 136 & * & 8 & 11 & 14 & 18 & 24 & 31\\
10 & 17 & 26 & 43 & 75 & 137 &*  & 8 & 11 & 14 & 18 & 24 & 31\\
\hline
\end{tabular}
\caption{\label{tab:HDG_tau2_gs} Example I. HDG with stabilization $\tau=1$: the number of iterations for multigrid with LU-SGS smoother as solver and preconditioner.}
\end{table}

\begin{table}[h!b!t!]
\centering
\begin{tabular}{|r|c|c|c|c|c|c||c|c|c|c|c|c|}
\hline
& \multicolumn{6}{c||}{MG as solver} & \multicolumn{6}{c|}{MG
with GMRES}\\
\cline{2-13}
\!\!\! $\p$ \!\!\!\! &  \multicolumn{6}{c||}{\!\!\scriptsize  Levels\!\!} &  \multicolumn{6}{c|}{\!\!\scriptsize Levels\!\!}\\
\hline
& \scriptsize 2 & \scriptsize 3 & \scriptsize 4 & \scriptsize 5 & \scriptsize 6 & \scriptsize 7 & \scriptsize 2 & \scriptsize 3 & \scriptsize 4 & \scriptsize 5 & \scriptsize 6 & \scriptsize 7 \\
\cline{2-13}
1 & 5 & 8 & 15 & 27 & 51 & 97 & 4 & 6 & 9 & 12 & 17 & 23\\
2 & 5 & 7 & 9 & 14 & 24 & 42 & 4 & 5 & 7 & 9 & 12 & 15\\
3 & 8 & 9 & 9 & 9 & 15 & 25 & 5 & 6 & 6 & 7 & 9 & 11\\
4 & 9 & 10 & 10 & 10 & 10 & 10 & 6 & 7 & 7 & 7 & 8 & 8\\
5 & 11 & 11 & 12 & 12 & 12 & 12 & 6 & 7 & 7 & 8 & 7 & 7\\
6 & 12 & 12 & 13 & 13 & 13 & 13 & 7 & 8 & 8 & 8 & 8 & 8\\
7 & 13 & 14 & 14 & 14 & 14 & 14 & 7 & 8 & 8 & 8 & 8 & 8\\
8 & 14 & 15 & 15 & 15 & 15 & 15 & 8 & 9 & 9 & 9 & 9 & 8\\
9 & 15 & 16 & 16 & 16 & 17 & 17  &8  &9  &9  &9  &9  &9 \\
10 &16  &17  &17  &17  &17  &17   &8  &9  &9  &9  &9  &9 \\
\hline
\end{tabular}
\caption{\label{tab:HDG_tau2_bj}
Example I. HDG with stabilization $\tau=1$: the number of iterations for multigrid with  block-Jacobi smoother as solver and preconditioner.}
    \vspace{-3mm}
\end{table}


%
%
In Tables \ref{tab:HDG_tau2_gs} and \ref{tab:HDG_tau2_bj}, we present
results for the multigrid solver and GMRES with multigrid preconditioner
when $\tau=1$ is used for the HDG discretization. With the
point-smoothers\footnote{We do not show results for point-Jacobi and
Chebyshev accelerated point-Jacobi smoothers as the behavior of the number of
iterations is very similar to that of the LU-SGS smoother.} (Table
\ref{tab:HDG_tau2_gs}), the number of iterations for the multigrid
solver scales like $\mc{O}(1/h)$ (since the number of iterations is nearly
doubled each time $h$ reduces by half).
Even for GMRES with multigrid as preconditioner, we do not
obtain $h$-scalability with point smoothers.
On the other hand,
the block-Jacobi smoother in Table \ref{tab:HDG_tau2_bj} for $\p \geq
4$ provides $h$-scalable results and with solution orders
greater than $5$ the results are exactly the same as those for
$\tau=1/h_{min}$. 
With respect to solution
orders, similar to  the case of $\tau=1/h_{min}$, we do not obtain perfect
$\p$-scalability with any smoother though the growth of the number of iterations is
very slow with block-Jacobi smoother.

\begin{table}[h!b!t!]
\centering
\begin{tabular}{|r|c|c|c|c|c|c||c|c|c|c|c|c|}
\hline
& \multicolumn{6}{c||}{MG as solver} & \multicolumn{6}{c|}{MG
with GMRES}\\
\cline{2-13}
\!\!\! $\p$ \!\!\!\! &  \multicolumn{6}{c||}{\!\!\scriptsize  Levels\!\!} &  \multicolumn{6}{c|}{\!\!\scriptsize Levels\!\!}\\
\hline
& \scriptsize 2 & \scriptsize 3 & \scriptsize 4 & \scriptsize 5 & \scriptsize 6 & \scriptsize 7 & \scriptsize 2 & \scriptsize 3 & \scriptsize 4 & \scriptsize 5 & \scriptsize 6 & \scriptsize 7 \\
\cline{2-13}
1 & 10 & 12 & 12 & 12 & 12 & 12 & 6 & 8 & 8 & 8 & 8 & 8\\
2 & 8 & 9 & 9 & 9 & 9 & 9 & 6 & 6 & 6 & 6 & 6 & 6\\
3 & 10 & 10 & 10 & 11 & 11 & 11 & 6 & 7 & 7 & 7 & 7 & 7\\
4 & 11 & 12 & 12 & 12 & 13 & 13 & 7 & 8 & 8 & 8 & 8 & 8\\
5 & 13 & 14 & 14 & 14 & 14 & 14 & 7 & 8 & 8 & 8 & 8 & 8\\
6 & 15 & 15 & 16 & 16 & 16 & 16 & 9 & 9 & 9 & 9 & 9 & 9\\
7 & 16 & 17 & 18 & 18 & 18 & 18 & 9 & 9 & 9 & 9 & 9 & 9\\
8 & 18 & 19 & 19 & 20 & 20 & 20  & 10 & 10 & 10 & 10 & 10 & 10\\
9 & 20 & 21 & 21 & 21 & 21 & 22  & 10 & 11 & 10 & 10 & 10 & 10\\
10 & 21 &22  &23  &23  &23  & 23   &11  &11  &11  &11  &11  & 11 \\
\hline
\end{tabular}
\caption{\label{tab:NIPG_J}
Example I. NIPG-H: the number of iterations for multigrid with  point-Jacobi smoother as solver and preconditioner.}
\end{table}

%
%

\begin{table}[h!b!t!]
\centering
\begin{tabular}{|r|c|c|c|c|c|c||c|c|c|c|c|c|}
\hline
& \multicolumn{6}{c||}{MG as solver} & \multicolumn{6}{c|}{MG
with GMRES}\\
\cline{2-13}
\!\!\! $\p$ \!\!\!\! &  \multicolumn{6}{c||}{\!\!\scriptsize  Levels\!\!} &  \multicolumn{6}{c|}{\!\!\scriptsize Levels\!\!}\\
\hline
& \scriptsize 2 & \scriptsize 3 & \scriptsize 4 & \scriptsize 5 & \scriptsize 6 & \scriptsize 7 & \scriptsize 2 & \scriptsize 3 & \scriptsize 4 & \scriptsize 5 & \scriptsize 6 & \scriptsize 7 \\
\cline{2-13}
1 & 4 & 4 & 4 & 4 & 4 & 4 & 3 & 4 & 4 & 4 & 4 & 5\\
2 & 4 & 4 & 4 & 5 & 5 & 5 & 4 & 4 & 4 & 4 & 4 & 4\\
3 & 4 & 5 & 5 & 5 & 5 & 5 & 4 & 4 & 4 & 5 & 5 & 4\\
4 & 5 & 5 & 5 & 5 & 5 & 6 & 4 & 4 & 5 & 5 & 5 & 5\\
5 & 5 & 6 & 6 & 6 & 6 & 6 & 4 & 5 & 5 & 5 & 5 & 5\\
6 & 6 & 6 & 6 & 6 & 6 & 6 & 5 & 5 & 5 & 5 & 5 & 5\\
7 & 6 & 7 & 7 & 7 & 7 & 7 & 5 & 5 & 5 & 5 & 5 & 5\\
8 & 7 & 7 & 7 & 7 & 7 & 7 & 5 & 6 & 6 & 6 & 6 & 6\\
9 & 7 & 7 & 8 & 8 & 8 & 8 & 5 & 6 & 6 & 6 & 6 & 6\\
10 & 7 &8  &8  &8  &8  &8   &6  &6  &6  &6  &6  &6 \\
\hline
\end{tabular}
\caption{\label{tab:NIPG_gs} Example I. NIPG-H: the number of iterations for multigrid with LU-SGS smoother as solver and preconditioner.}
    \vspace{-4mm}
\end{table}

\begin{table}[h!b!t!]
\centering
\begin{tabular}{|r|c|c|c|c|c|c||c|c|c|c|c|c|}
\hline
& \multicolumn{6}{c||}{MG as solver} & \multicolumn{6}{c|}{MG
with GMRES}\\
\cline{2-13}
\!\!\! $\p$ \!\!\!\! &  \multicolumn{6}{c||}{\!\!\scriptsize  Levels\!\!} &  \multicolumn{6}{c|}{\!\!\scriptsize Levels\!\!}\\
\hline
& \scriptsize 2 & \scriptsize 3 & \scriptsize 4 & \scriptsize 5 & \scriptsize 6 & \scriptsize 7 & \scriptsize 2 & \scriptsize 3 & \scriptsize 4 & \scriptsize 5 & \scriptsize 6 & \scriptsize 7 \\
\cline{2-13}
1 & 5 & 7 & 8 & 8 & 8 & 8 & 4 & 6 & 6 & 6 & 6 & 6\\
2 & 6 & 6 & 6 & 6 & 7 & 7 & 5 & 5 & 5 & 5 & 5 & 5\\
3 & 7 & 8 & 8 & 8 & 8 & 8 & 5 & 6 & 6 & 6 & 6 & 6\\
4 & 9 & 9 & 10 & 10 & 10 & 10 & 6 & 7 & 7 & 7 & 7 & 7\\
5 & 10 & 11 & 11 & 11 & 11 & 11 & 6 & 7 & 7 & 7 & 7 & 7\\
6 & 11 & 12 & 12 & 12 & 12 & 12 & 7 & 8 & 8 & 8 & 8 & 7\\
7 & 12 & 13 & 13 & 13 & 13 & 13 & 7 & 8 & 8 & 8 & 8 & 8\\
8 & 13 & 14 & 14 & 14 & 14 & 14 & 7 & 8 & 8 & 8 & 8 & 8\\
9 & 14 & 15 & 15 & 15 & 15 & 15  & 8 & 9 & 9 & 9 & 9 & 8\\
10 & 15 &16  &16  &16  &16  &16   &8  &9  &9  &9  &9  &9 \\
\hline
\end{tabular}
\caption{\label{tab:NIPG_bj}
Example I. NIPG-H: the number of iterations for multigrid with  block-Jacobi smoother as solver and preconditioner.}
\end{table}



We present in Tables \ref{tab:NIPG_J}-\ref{tab:NIPG_bj} the results
for the hybridized NIPG-H scheme with $\tau=(\p+1)(\p+2)/h_{min}$. As can
be seen, the number of iterations for both multigrid as a solver and as a preconditioner is 
 similar to those with HDG using $\tau=1/h_{min}$.  The results
using IIPG-H and SIPG-H are similar and hence are  omitted for brevity.  For all
of the IPDG schemes, using the Chebyshev acceleration increases the
number of iterations and we do not include these results here.

Finally, we test the performance of multigrid without the local correction operator in step 3 of Algorithm
 \ref{al:multigrid}. The results are shown for HDG with $\tau=1/h_{min}$ and block-Jacobi smoother in Table \ref{tab:HDG_tau1_bj_without_local}. As can be seen, for $\p\leq3$ multigrid as solver and also as a preconditioner fails to
 converge for fine mesh sizes. For $\p\geq4$ we see scalable results although the number of iterations are slightly
 more than that in Table \ref{tab:HDG_tau1_bj} with local correction. We also see an odd-even behavior, with even orders
 giving better iteration counts than odd ones. For $\tau=1$, we observed similar results and hence not shown. For the
 hybridized IPDG schemes, with NIPG-H and IIPG-H from $\p\geq2$ onwards we observed scalable results whereas for SIPG-H
 for all orders we obtained scalability without local correction. The iteration counts are again slightly more than the ones obtained with local
 correction. With increase in order, since block-Jacobi smoother gets stronger, even without the local correction operator it
 is sufficient to give scalable results. However, in order to provide
 more robustness and since the local correction operator is inexpensive, in the following test cases we always include
 it in the step-3 of Algorithm \ref{al:multigrid}.

\begin{table}[h!b!t!]
\centering
\begin{tabular}{|r|c|c|c|c|c|c||c|c|c|c|c|c|}
\hline
& \multicolumn{6}{c||}{MG as solver} & \multicolumn{6}{c|}{MG
with GMRES}\\
\cline{2-13}
\!\!\! $\p$ \!\!\!\! &  \multicolumn{6}{c||}{\!\!\scriptsize  Levels\!\!} &  \multicolumn{6}{c|}{\!\!\scriptsize Levels\!\!}\\
\hline
& \scriptsize 2 & \scriptsize 3 & \scriptsize 4 & \scriptsize 5 & \scriptsize 6 & \scriptsize 7 & \scriptsize 2 & \scriptsize 3 & \scriptsize 4 & \scriptsize 5 & \scriptsize 6 & \scriptsize 7 \\
\cline{2-13}
1 & * & * & * & * & * & * & 8 & 18 & 56 & * & * & *\\
2 & 8 & 13 & 21 & 26 & 27 & 28 & 6 & 7 & 9 & 11 & 11 & 11\\
3 & 16 & 60 & * & * & * & * & 7 & 11 & 20 & 39 & 97 & *\\
4 & 10 & 11 & 12 & 13 & 13 & 13 & 7 & 7 & 8 & 8 & 8 & 8\\
5 & 11 & 12 & 12 & 12 & 12 & 12 & 6 & 8 & 8 & 8 & 8 & 7\\
6 & 12 & 12 & 13 & 13 & 13 & 13 & 7 & 8 & 8 & 8 & 8 & 8\\
7 & 15 & 20 & 23 & 25 & 26 & 26 & 9 & 10 & 11 & 11 & 11 & 11\\
8 & 14 & 15 & 15 & 15 & 15 & 15 & 8 & 9 & 9 & 9 & 9 & 8\\
9 & 15 & 16 & 16 & 16 & 17 & 17  & 8 & 9 & 9 & 9 & 9 & 9\\
10 &18  &19  &19  &20  &20  &20   &9  &10  &10  &10  &10  & 10\\
\hline
\end{tabular}
\caption{\label{tab:HDG_tau1_bj_without_local} Example I. HDG with stabilization $\tau=1/h_{min}$: the number of iterations for multigrid with block-Jacobi smoother as solver and preconditioner without local correction.}
    \vspace{-5mm}
\end{table}

In summary, for HDG with $\tau=1/h_{min}$ and multigrid preconditioned GMRES, almost perfect
$hp$-scalability has been observed for all the
smoothers, while for $\tau = 1$ this scalability holds with only block
smoothers together with high solution orders. 
All the hybridized IPDG schemes give scalable results with
$\tau=(\p+1)(\p+2)/h_{min}$ and the number of iterations is very
similar to that of HDG with $\tau=1/h_{min}$. Of all the smoothers
considered in this paper, LU-SGS seems to be the best
smoother in terms of the number of iterations. However, block-Jacobi
smoother, the second best, is more convenient from the implementation
point of view (especially on parallel computing sytems), and for that reason we  show numerical results only for block-Jacobi smoother from now on.
Since the Chebyshev acceleration does not
seem to improve the performance of multigrid by a significant margin, we will not
 use it in the subsequent sections. We would like to mention that the performance of our multigrid approach in terms of iteration counts is better or at least comparable to other existing  multigrids  proposed in the literature \cite{
 Gopalakrishnan09aconvergent,cockburn2014multigrid,kronbichler2018performance,chen2015auxiliary}. Since
 the time to solution depends on many other factors such as the  choice of smoother,
the  number of smoothing steps, and the architecture of the machine we are not able to comment on that aspect at this moment. However, in
  future work we plan to carry out a detailed study between different multigrid algorithms for hybridized methods with respect to simulation of challenging problems.

\subsection{Example II: Unstructured mesh}
\seclab{exp2}
The goal of this section is to test the robustness of the multigrid
algorithm for a highly unstructured mesh. To that end consider the
mesh in Figure \ref{fat_box}, which consists of $6699$ simplices with
an order of magnitude variation in the diameter of elements
i.e. $h_{max} \approx 10 h_{min}$.
Unlike structured meshes, for unstructured meshes with local
clustering of elements (local refinements) it is not straightforward
 to select the number of levels and the ``best'' coarsening
strategy that well balances the number of elements (in the finer level) for
each macro-element. Ideally, given a number of levels we wish to minimize the  number of multigrid/GMRES iterations. Here, we compare two
coarsening strategies with the same number of levels and study the
performance of multigrid as a solver and as a preconditioner. In our
future work, we will explore  the coarsening strategies similar to the
ones in
\cite{venkatakrishnan1993unstructured,venkatakrishnan1995agglomeration}.
Figures \figref{cs1} and
\figref{cs2} show different levels corresponding to the two coarsening
strategies. The total number of levels in both the strategies
is seven and the last level corresponds to the original fine mesh as shown in Figure 
\ref{fat_box}. For solution orders $\p>1$, again, we first restrict the residual
to $\p=1$ and then carry out the geometric coarsening.
There are seven macro-elements in the first level for coarsening strategy $1$, and four macro-elements for strategy $2$.

    We take zero Dirichlet boundary condition, ${\bf K}={\bf I}$ ({\bf I}
is the identity), and $f=1$ for this example.  We now study
the performance of HDG and hybridized IPDG schemes. In Table
\ref{tab:HDG}, we compare the number of iterations taken for HDG with
$\tau_1=1$ and $\tau_2=1/h_{min}$ and solution orders
$\p=1,\hdots,8$. As can be seen, for $\p\in \LRc{1,2,3},$ using
$\tau_1$ takes less number of iterations, and for $\p \ge 4$ both stabilization parameters require almost the same iteration counts. 
For coarsening strategy 2,
since the coarsening happens slightly more aggressive (i.e. less macro
elements at any level) the iteration counts in general are 
higher. However, using multigrid as a preconditioner for GMRES the
difference in the iteration counts for the two strategies is negligible.

In Table \ref{tab:NIPG}, we consider the NIPG-H scheme
with 
$\tau_1=1/h_{min}$ and $\tau_2=(\p+1)(\p+2)/h_{min}$. As can be seen, the multigrid solver diverges for many values of  solution order $\p$. The iterations
for $\tau_1$ are less than that for
$\tau_2$. 
We observe that the iteration counts of multigrid preconditioned GMRES with $\tau_1$, except for solution orders $\p = 2$ and $\p = 3$ (which require more iterations),  are similar to those for HDG. 
When multigrid is used as a preconditioner, except for solution orders equal to two and three,  both the coarsening strategies give similar iteration counts with $\tau_1$. It turns out that the iteration counts for IIPG-H and SIPG-H  are higher than that of HDG and NIPG-H and hence are
omitted for brevity.


Within the settings of this example we conclude that
{\em multigrid}, both
{\em as a solver and as a preconditioner, is more effective  for
HDG than for hybridized IPDG schemes}. It is also relatively more robust with respect to the coarsening strategies and values of stabilization parameter when used as a preconditioner.

\begin{table}[h!b!t!]
\centering
\begin{tabular}{|r|c c|c c||c c|c c|}
\hline
& \multicolumn{4}{c||}{MG as solver} & \multicolumn{4}{c|}{MG
with GMRES}\\
\cline{2-9}
\!\!\! $\p$ \!\!\!\! &  \multicolumn{2}{c|}{\!\!\scriptsize  coarsening strategy 1\!\!} &  \multicolumn{2}{c||}{\!\!\scriptsize coarsening strategy 2\!\!} & \multicolumn{2}{c|}{\!\!\scriptsize coarsening strategy 1\!\!} &  \multicolumn{2}{c|}{\!\!\scriptsize coarsening strategy 2\!\!}\\
\hline
& \scriptsize $\tau_1$ & \scriptsize $\tau_2$ & \scriptsize $\tau_1$ & \scriptsize $\tau_2$ & \scriptsize $\tau_1$ & \scriptsize $\tau_2$ & \scriptsize $\tau_1$ & \scriptsize $\tau_2$ \\
\cline{2-9}
1 & 25 & * & 41 & 108 & 10 & 21 & 13 & 18\\
2 & 20 & 63 & 33 & 37 & 10 & 12 & 12 & 13\\
3 & 23 & 24 & 39 & 41 & 10 & 11 & 13 & 14\\
4 & 27 & 27 & 43 & 45 & 11 & 11 & 14 & 14\\
5 & 30 & 30 & 48 & 49 & 11 & 12 & 15 & 15\\
6 & 33 & 33 & 51 & 52 & 12 & 12 & 15 & 15\\
7 & 36 & 36 & 55 & 55 & 13 & 13 & 16 & 16\\
8 & 38 & 38 & 58 & 58 & 13 & 13 & 16 & 16\\
\hline
\end{tabular}
\caption{\label{tab:HDG}
Example II. HDG with $\tau_1=1$, $\tau_2=1/h_{min}$: the number of iterations for multigrid as solver and preconditioner for both coarsening strategies.}
\end{table}
\begin{table}[h!b!t!]
\centering
\begin{tabular}{|r|c c|c c||c c|c c|}
\hline
& \multicolumn{4}{c||}{MG as solver} & \multicolumn{4}{c|}{MG
with GMRES}\\
\cline{2-9}
\!\!\! $\p$ \!\!\!\! &  \multicolumn{2}{c|}{\!\!\scriptsize  coarsening strategy 1\!\!} &  \multicolumn{2}{c||}{\!\!\scriptsize coarsening strategy 2\!\!} & \multicolumn{2}{c|}{\!\!\scriptsize coarsening strategy 1\!\!} &  \multicolumn{2}{c|}{\!\!\scriptsize coarsening strategy 2\!\!}\\
\hline
& \scriptsize $\tau_1$ & \scriptsize $\tau_2$ & \scriptsize $\tau_1$ & \scriptsize $\tau_2$ & \scriptsize $\tau_1$ & \scriptsize $\tau_2$ & \scriptsize $\tau_1$ & \scriptsize $\tau_2$ \\
\cline{2-9}
1 & 16 & * & 29 & * & 9 & 133 & 11 & 112\\
2 & * & * & * & * & 136 & * & * & *\\
3 & * & * & * & * & 42 & 42 & 49 & 31\\
4 & 22 & * & 37 & * & 10 & 71 & 13 & 62\\
5 & 25 & 67 & 41 & 95 & 11 & 16 & 14 & 19\\
6 & 28 & * & 46 & * & 11 & 28 & 14 & 27\\
7 & 31 & 71 & 49 & 99 & 12 & 17 & 15 & 21\\
8 & 34 & 121 & 53 & 106 & 12 & 20 & 15 & 21\\
\hline
\end{tabular}
\caption{\label{tab:NIPG}
Example II. NIPG-H with $\tau_1=1/h_{min}$, $\tau_2=(\p+1)(\p+2)/h_{min}$: the number of iterations for multigrid as solver and preconditioner for both coarsening strategies.}
\end{table}
\subsection{Example III: Smoothly varying permeability}
\seclab{exp3}
In this example we consider a smoothly varying permeability of the form
\[
    {\bf K} = \kappa {\bf I},
\]
where $\kappa =1+0.5\sin(2\pi x)\cos(3\pi y)$. 
The domain considered is a circle and we take the exact solution again to be of the form $\pres=x y e^{x^2y^3}$. The forcing and the Dirichlet boundary condition are chosen 
corresponding to the exact solution.
The number of levels and the corresponding number of triangular
elements 
in the multigrid hierarchy are shown in Table \ref{tab:level_elements_circle}. 
\begin{table}[h!b!t!]
\begin{center}
\begin{tabular}{ | r || c c  c  c  c  c | }
\hline
Levels &  2 & 3 & 4 & 5 & 6 & 7 \\
\hline
Elements & 84 & 348 & 1500 & 6232 & 25344 & 102160 \\
\hline
\end{tabular}
\caption{\label{tab:level_elements_circle}
Example III: The multigrid hierarchy.}
    \vspace{-5mm}
\end{center}
\end{table}

\begin{table}[h!b!t!]
\centering
\begin{tabular}{|r|c|c|c|c|c|c||c|c|c|c|c|c|}
\hline
& \multicolumn{6}{c||}{MG as solver} & \multicolumn{6}{c|}{MG
with GMRES}\\
\cline{2-13}
\!\!\! $\p$ \!\!\!\! &  \multicolumn{6}{c||}{\!\!\scriptsize  Levels\!\!} &  \multicolumn{6}{c|}{\!\!\scriptsize Levels\!\!}\\
\hline
& \scriptsize 2 & \scriptsize 3 & \scriptsize 4 & \scriptsize 5 & \scriptsize 6 & \scriptsize 7 & \scriptsize 2 & \scriptsize 3 & \scriptsize 4 & \scriptsize 5 & \scriptsize 6 & \scriptsize 7 \\
\cline{2-13}
1 & 18 & 20 & 23 & 21 & 24 & 27 & 9 & 10 & 11 & 11 & 11 & 12\\
2 & 13 & 15 & 18 & 17 & 19 & 20 & 9 & 9 & 10 & 10 & 11 & 11\\
3 & 16 & 18 & 21 & 20 & 23 & 24 & 10 & 10 & 11 & 11 & 11 & 12\\
4 & 18 & 21 & 23 & 23 & 25 & 27 & 10 & 11 & 12 & 12 & 12 & 12\\
5 & 20 & 24 & 26 & 25 & 28 & 29 & 11 & 12 & 12 & 12 & 13 & 13\\
6 & 22 & 27 & 28 & 28 & 31 & 32 & 11 & 13 & 13 & 13 & 13 & 13\\
7 & 24 & 30 & 30 & 30 & 33 & 34 & 12 & 13 & 13 & 13 & 14 & 14\\
8 & 26 & 32 & 32 & 32 & 35 & 36  & 12 & 13 & 14 & 14 & 14 & 14\\
\hline
\end{tabular}
\caption{\label{tab:HDG_smooth_k_tau1}
Example III. HDG with stabilization
  $\tau=1$: the number of iterations for multigrid as solver and preconditioner.}
\end{table}
Table \ref{tab:HDG_smooth_k_tau1} shows the number of multigrid and GMRES iterations for HDG with $\tau=1$.
Note that perfect $h$-scalability with multigrid as a solver is not observed, and this is 
mainly due to the fact that  the number of elements between successive levels is
not exactly increased/decreased by a constant\footnote{This, as argued before, is not trivial for unstructured meshes.}. Again, we see a little growth in the number of
iterations as $\p$ increases for the multigrid solver, while  for GMRES with multigrid preconditioner we
obtain almost perfect $hp$-scalability. We
have also conducted numerical examples with $\tau=1/h_{min}$, $\tau=\kappa/h_{min}$
and observed that the iteration counts (not shown here) are similar to those with $\tau=1$.

Next we consider the IIPG-H scheme with 
$\tau=\kappa(\p+1)(\p+2)/h_{min}$. Table~\ref{tab:IIPG_smooth_k} shows that
the iteration counts are comparable with  HDG results in  Table \ref{tab:HDG_smooth_k_tau1} for 
GMRES with multigrid preconditioner.
The multigrid solver for IIPG-H, on the other hand, has slightly less
iteration counts.
An extra penalization factor of $1.5$ is necessary to obtain well-posed linear systems for SIPG-H, but otherwise
the corresponding results for NIPG-H and SIPG-H are similar and hence are omitted.

\begin{table}[h!b!t!]
  \centering
  \begin{tabular}{|r|c|c|c|c|c|c||c|c|c|c|c|c|}
    \hline
    & \multicolumn{6}{c||}{MG as solver} & \multicolumn{6}{c|}{MG
      with GMRES}\\
    \cline{2-13}
    \!\!\! $\p$ \!\!\!\! &  \multicolumn{6}{c||}{\!\!\scriptsize  Levels\!\!} &  \multicolumn{6}{c|}{\!\!\scriptsize Levels\!\!}\\
\hline
      & \scriptsize 2 & \scriptsize 3 & \scriptsize 4 & \scriptsize 5 & \scriptsize 6 & \scriptsize 7 & \scriptsize 2 & \scriptsize 3 & \scriptsize 4 & \scriptsize 5 & \scriptsize 6 & \scriptsize 7 \\
 \cline{2-13}
      1 & 14 & 15 & 17 & 16 & 18 & 20 & 8 & 9 & 10 & 10 & 10 & 10\\
      2 & 12 & 13 & 16 & 15 & 17 & 18 & 8 & 9 & 9 & 9 & 10 & 10\\
      3 & 14 & 16 & 18 & 18 & 20 & 21 & 9 & 10 & 10 & 10 & 11 & 11\\
      4 & 17 & 19 & 21 & 21 & 23 & 24 & 10 & 10 & 11 & 11 & 12 & 12\\
      5 & 19 & 21 & 23 & 23 & 25 & 26 & 11 & 11 & 12 & 12 & 12 & 12\\
      6 & 21 & 24 & 26 & 25 & 28 & 29 & 11 & 12 & 12 & 12 & 13 & 13\\
      7 & 22 & 26 & 27 & 27 & 30 & 31 & 11 & 13 & 13 & 13 & 13 & 13\\
      8 & 24 & 29 & 29 & 29 & 32 & 33 & 12 & 13 & 13 & 13 & 13 & 14\\
\hline
  \end{tabular}
    \caption{\label{tab:IIPG_smooth_k} Example III. IIPG-H with stabilization $\tau=\kappa(\p+1)(\p+2)/h_{min}$: the number of iterations for multigrid as solver and preconditioner.}
    \vspace{-4mm}
\end{table}

\subsection{Example IV: Highly discontinuous permeability}
\seclab{exp4}
In this example, we test the robustness of the multigrid solver on
a highly discontinuous and spatially varying (six orders of magnitude) permeability field given by ${\bf K} = \kappa
{\bf I}$ with
\[ \kappa = \begin{cases}
10^6, & (x,y) \in (0,0.56)\times (0,0.56) \quad \text{or} \quad (x,y) \in (0.56,1)\times (0.56,1),\\
1, & \text{otherwise},
\end{cases}
\]
in a unit square.

We choose zero Dirichlet boundary condition and $f=1$. In Table
\ref{tab:HDG_discont_k_tau1}, we study the performance of multigrid as
a solver and as a preconditioner for GMRES using HDG with
$\tau=\kappa/h_{min}$. The number of elements and the number of levels
are same as in example I. Similar to the previous examples,
almost perfect $h$- and $\p$-scalabililties are observed with
multigrid preconditioned GMRES, whereas  the
number of iterations increases with the solution order  $\p$ for the multigrid solver.
\begin{table}[h!b!t!]
\centering
\begin{tabular}{|r|c|c|c|c|c|c||c|c|c|c|c|c|}
\hline
& \multicolumn{6}{c||}{MG as solver} & \multicolumn{6}{c|}{MG
with GMRES}\\
\cline{2-13}
\!\!\! $\p$ \!\!\!\! &  \multicolumn{6}{c||}{\!\!\scriptsize  Levels\!\!} &  \multicolumn{6}{c|}{\!\!\scriptsize Levels\!\!}\\
\hline
& \scriptsize 2 & \scriptsize 3 & \scriptsize 4 & \scriptsize 5 & \scriptsize 6 & \scriptsize 7 & \scriptsize 2 & \scriptsize 3 & \scriptsize 4 & \scriptsize 5 & \scriptsize 6 & \scriptsize 7 \\
\cline{2-13}
1 & 6 & 8 & 8 & 9 & 9 & 9 & 2 & 5 & 5 & 5 & 5 & 5\\
2 & 6 & 7 & 9 & 10 & 10 & 11 & 3 & 4 & 5 & 5 & 5 & 5\\
3 & 8 & 9 & 10 & 10 & 10 & 10 & 3 & 5 & 6 & 6 & 5 & 5\\
4 & 10 & 11 & 11 & 11 & 11 & 11 & 4 & 6 & 6 & 6 & 6 & 5\\
5 & 11 & 13 & 13 & 13 & 13 & 13 & 4 & 6 & 7 & 7 & 6 & 6\\
6 & 13 & 14 & 14 & 14 & 14 & 14 & 4 & 6 & 7 & 7 & 7 & 6\\
7 & 14 & 15 & 16 & 16 & 15 & 15 & 5 & 7 & 7 & 7 & 7 & 7\\
8 & 15 & 16 & 17 & 17 & 17 & 17 & 5 & 7 & 8 & 8 & 7 & 7\\
\hline
\end{tabular}
\caption{\label{tab:HDG_discont_k_tau1} Example IV. HDG with stabilization
  $\tau=\kappa/h_{min}$: the number of iterations for multigrid as solver and preconditioner.}
\end{table}
For $\tau=1/h_{min}$ and $\tau=1$ the multigrid solver takes too many iterations to converge for
finer meshes, and for that reason we report the iteration counts only for multigrid preconditioned
GMRES in Table \ref{tab:HDG_discont_k_tau23}.  
\begin{table}[h!b!t!]
\centering
\begin{tabular}{|r|c|c|c|c|c|c||c|c|c|c|c|c|}
\hline
& \multicolumn{6}{c||}{$\tau=1/h_{min}$} & \multicolumn{6}{c|}{$\tau=1$}\\
\cline{2-13}
\!\!\! $\p$ \!\!\!\! &  \multicolumn{6}{c||}{\!\!\scriptsize  Levels\!\!} &  \multicolumn{6}{c|}{\!\!\scriptsize Levels\!\!}\\
\hline
& \scriptsize 2 & \scriptsize 3 & \scriptsize 4 & \scriptsize 5 & \scriptsize 6 & \scriptsize 7 & \scriptsize 2 & \scriptsize 3 & \scriptsize 4 & \scriptsize 5 & \scriptsize 6 & \scriptsize 7 \\
\cline{2-13}
1 & 3 & 5 & 6 & 8 & 11 & 5 & 2 & 5 & 8 & 11 & 15 & 15\\
2 & 3 & 5 & 6 & 5 & 5 & 5 & 3 & 5 & 6 & 7 & 8 & 11\\
3 & 3 & 5 & 6 & 6 & 5 & 5 & 3 & 4 & 5 & 6 & 7 & 8\\
4 & 4 & 5 & 6 & 6 & 6 & 5 & 4 & 5 & 6 & 6 & 7 & 7\\
5 & 4 & 5 & 7 & 7 & 6 & 6 & 4 & 5 & 7 & 7 & 6 & 6\\
6 & 4 & 6 & 7 & 7 & 7 & 6 & 4 & 6 & 7 & 7 & 7 & 6\\
7 & 4 & 6 & 7 & 7 & 7 & 7 & 4 & 6 & 7 & 7 & 7 & 7\\
8 & 4 & 6 & 8 & 8 & 7 & 7 & 4 & 6 & 8 & 8 & 7 & 7\\
\hline
\end{tabular}
\caption{\label{tab:HDG_discont_k_tau23}
Example IV. HDG with stabilizations $\tau=1/h_{min}$ and $\tau=1$: the number of iterations for multigrid preconditioned GMRES.}
\end{table}
Clearly, the number of GMRES iterations is almost insensitive to the
values of the stabilization parameter $\tau$, and the results for
$\tau=1/h_{min}$ in columns $2-7$ of Table
\ref{tab:HDG_discont_k_tau23} are very similar to columns $8-13$ of
Table \ref{tab:HDG_discont_k_tau1} for $\tau=\kappa/h_{min}$.
 Columns $8-13$ of Table
 \ref{tab:HDG_discont_k_tau23} shows that, again with $\tau=1$, high-order solutions provide
 both $h$- and $\p$-scalabilities and the results in these cases are similar to those with $\tau=1/h_{min}$
(see the last four rows of Table  \ref{tab:HDG_discont_k_tau23}). 


Table~\ref{tab:SIPG_discont_k} shows the performance of our multigrid algorithm
 as a solver and as a preconditioner for the SIPG-H scheme with  $\tau=\kappa(\p+1)(\p+2)/h_{min}$.
\begin{table}[h!b!t!]
\centering
\begin{tabular}{|r|c|c|c|c|c|c||c|c|c|c|c|c|}
\hline
& \multicolumn{6}{c||}{MG as solver} & \multicolumn{6}{c|}{MG
with GMRES}\\
\cline{2-13}
\!\!\! $\p$ \!\!\!\! &  \multicolumn{6}{c||}{\!\!\scriptsize  Levels\!\!} &  \multicolumn{6}{c|}{\!\!\scriptsize Levels\!\!}\\
\hline
& \scriptsize 2 & \scriptsize 3 & \scriptsize 4 & \scriptsize 5 & \scriptsize 6 & \scriptsize 7 & \scriptsize 2 & \scriptsize 3 & \scriptsize 4 & \scriptsize 5 & \scriptsize 6 & \scriptsize 7 \\
\cline{2-13}
1 & 8 & 9 & 10 & 11 & 11 & 12 & 3 & 6 & 6 & 6 & 6 & 6\\
2 & 6 & 8 & 9 & 9 & 10 & 10 & 4 & 5 & 5 & 5 & 5 & 5\\
3 & 9 & 10 & 10 & 10 & 10 & 11 & 4 & 6 & 6 & 6 & 6 & 5\\
4 & 10 & 11 & 12 & 12 & 12 & 12 & 4 & 6 & 7 & 6 & 6 & 6\\
5 & 12 & 13 & 13 & 13 & 13 & 13 & 4 & 6 & 7 & 7 & 7 & 6\\
6 & 13 & 14 & 15 & 15 & 15 & 15 & 5 & 7 & 7 & 7 & 7 & 7\\
7 & 14 & 16 & 16 & 16 & 16 & 16 & 5 & 7 & 8 & 8 & 7 & 7\\
8 & 15 & 17 & 17 & 17 & 17 & 17 & 5 & 7 & 8 & 8 & 7 & 7\\
\hline
\end{tabular}
\caption{\label{tab:SIPG_discont_k}
Example IV. SIPG-H with $\tau=\kappa(\p+1)(\p+2)/h_{min}$: the number of iterations for multigrid as solver and preconditioner.}
    \vspace{-4mm}
\end{table}
The results for NIPG-H and IIPG-H are almost identical and are omitted.
As can be observed,  as a preconditioner for GMRES the multigrid algorithm is  both $h$- and $\p$-scalable for all the hybridized DG methods.

\subsection{Example V: SPE10 test case}
\seclab{exp5}

In this example, we consider the benchmark problem Model 2 from the
Tenth Society of Petroleum Comparative Solution Project (SPE10)
\cite{christie2001tenth}. We consider the permeability field ${\bf K} = \kappa
{\bf I}$, where $\kappa$ corresponding to the $75$th layer is shown in left of Figure \ref{perm_pressure}. 
The permeability field varies by six orders of magnitude and is highly
heterogeneous which gives rise to extremely complex velocity
fields. The domain is $1200\times2200$ $[ft^2]$. The mesh has
$60\times220$ quadrilateral elements and the element edges align with
the discontinuties in the permeability. We choose $f=0$, which
corresponds to no source or sink. For the boundary conditions we take
the pressure (\pres) on the left and right faces to be $1$ and $0$,
respectively. On the top and bottom faces no flux boundary condition
$\ub\cdot\n=0$ is applied. From extensive numerical examples we
have observed that the multigrid hierarchy with seven levels results in
the least number of GMRES iterations. From numerical examples I-IV, we see that 
HDG method is relatively the most robust and scalable. In addition,
it provides simultaneous approximations for both
velocity and pressure. Thus, we consider only HDG for this example.

Since multigrid as a solver either converges
very slowly or diverges for a number of
cases in this example, we report results exclusively for the  multigrid preconditioned GMRES.
The pressure field is shown in the right of Figure~\ref{perm_pressure} for $\p=1$.
In Table \ref{tab:tau_1byh_1em9} are  the number of iterations for solution orders $\p\in \LRc{1,2,3,4}$ with
 $\tau=1/h_{min}$ and $\tau=1$.
The
second row shows that,  for $\p=2$, the iterations are much
larger compared to other solution orders. 
By the time of writing, we have not yet found the reason for this
behaviour.
For other solution orders, 
using $\tau = 1$ results in more iterations for $\p=1$ and less for $\p=3,4$ compared to $\tau=1/h_{min}$.
Again, we like to emphasize  that in this example we completely avoid upscaling of the permeability field as our multigrid algorithm is based only on the fine scale DtN maps.

The results in  columns two and three of Table \ref{tab:tau_1byh_1em9} correspond to  geometrically increasing smoothing steps with respect to levels as explained in Section \secref{smoothing}.
In columns four and five, we simply take two pre- and post-smoothing steps in all levels and, as can be seen, 
the iteration counts  are marginally different compared to those resulted from the increasing smoothing steps.
This implies that we can achieve similar
accuracy with less computational cost using constant (here two) pre- and post-smoothing steps in all levels.

\begin{figure}[h!t!b!]
  \subfigure[Permeability field]{
    \includegraphics[trim=1.9cm 6.5cm 2.9cm 6.4cm,clip=true,width=0.5\textwidth]{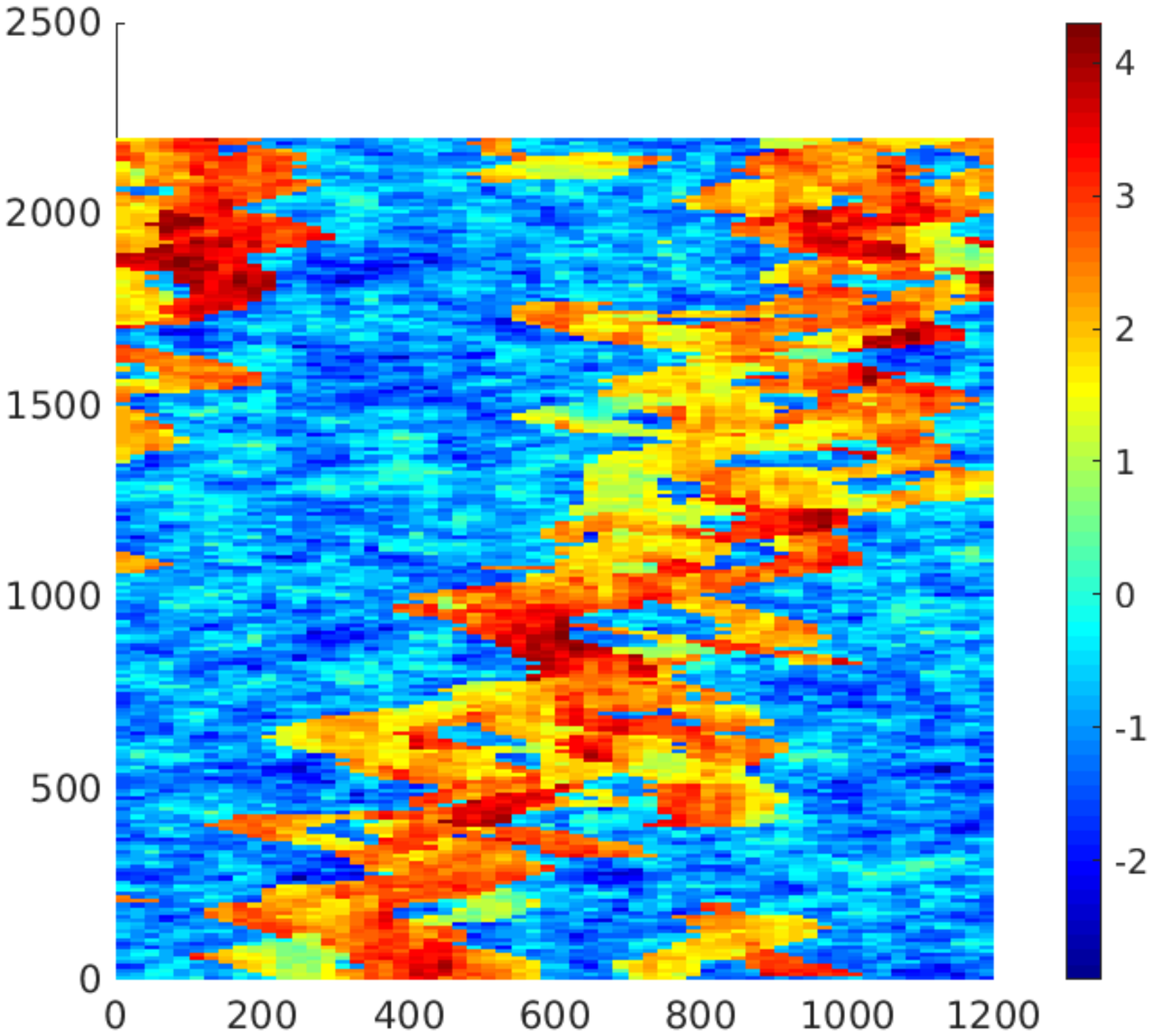}
    \label{spe_perm}
  }
  \subfigure[Pressure field]{
    \includegraphics[trim=1.9cm 6.5cm 2.6cm 6.4cm,clip=true,width=0.5\textwidth]{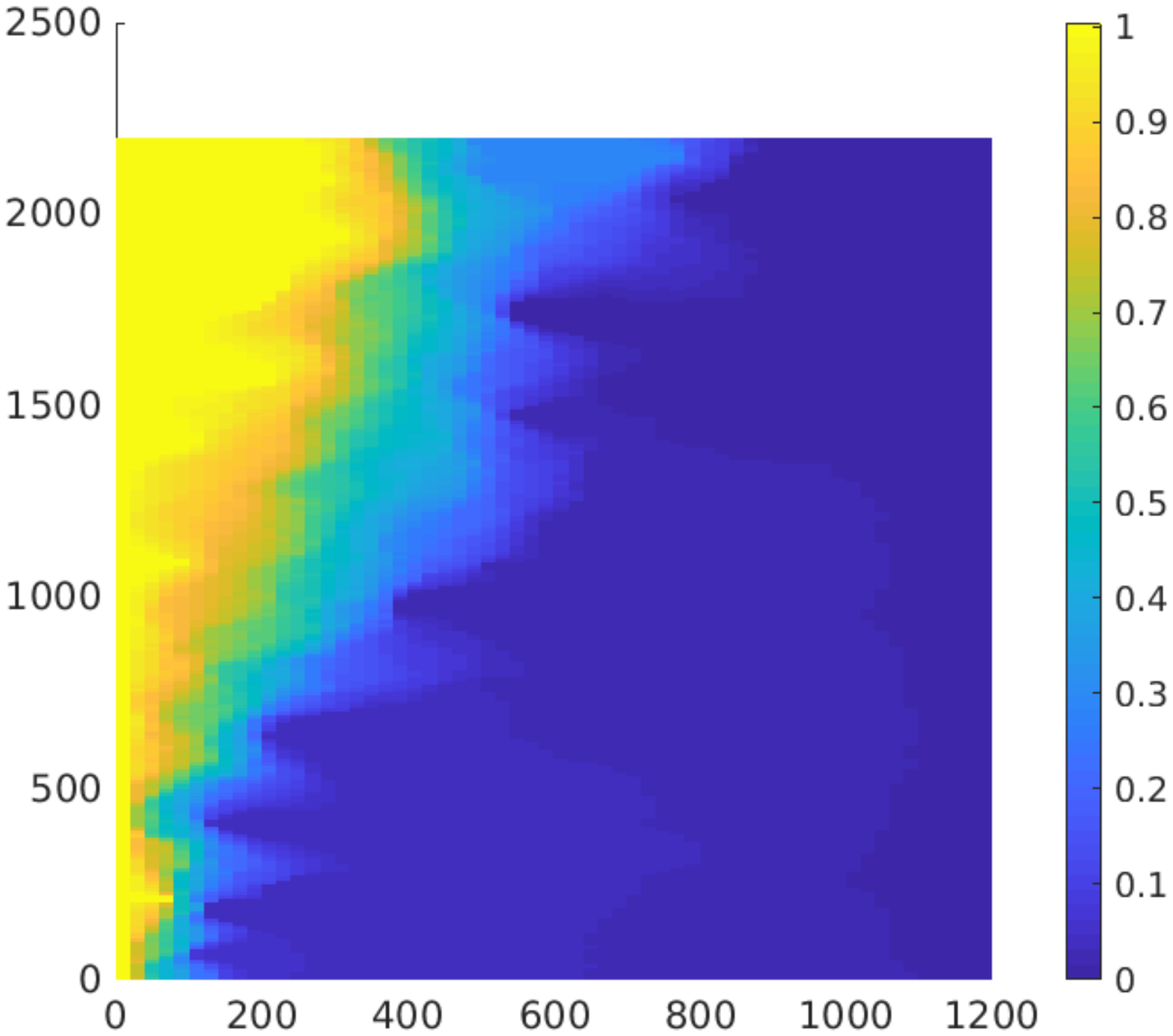}
    \label{spe_pressure}
  }
    \caption{Example V: Permeability ($\kappa$)
    in log scale (left) and pressure field (right) using solution order $\p=1$. }
 \vspace{-7mm}
  \label{perm_pressure}
\end{figure}


\begin{table}[h!b!t!]
 \begin{center}
 \begin{tabular}{ | r || c | c | c | c | }
 \hline
     \multirow{2}{*}{$\p$} & \multicolumn{2}{|c|}{Varying smoothing} & \multicolumn{2}{|c|}{2 pre- and post-smoothing} \\
     \cline{2-5}
     & $\tau=1/h_{min}$ & $\tau=1$ & $\tau=1/h_{min}$ & $\tau=1$\\
 \hline
     1 & 47 & 78 & 51 & 80 \\
 \hline
 2 & 81 & * & 72 & * \\
 \hline
 3 & 58 & 40 & 56 & 48 \\
 \hline
 4 & 48 & 39 & 49 & 45 \\
 \hline
 \end{tabular}
\caption{Example V. HDG with stabilizations $\tau=1/h_{min}$, $\tau=1$: the number of iterations for multigrid preconditioned GMRES with variable smoothing and constant smoothing.}
\label{tab:tau_1byh_1em9}
     \vspace{-5mm}
 \end{center}
 \end{table}

\subsection{Example VI: Multinumerics}
\seclab{multinumerics}

The goal of this section is to demonstrate that the proposed  multigrid algorithm can handle different fine scale DtN maps (multinumerics) on unstructured grids without explicit upscaling for macro-elements. 
In the following, we choose a combination of RT1 and NIPG as a representative for 
    other combinations of hybridized methods. We expect similar results using other combinations, as our multigrid algorithm operates on the
    trace system. Also we would like to point out that the other multigrid algorithms for hybridized methods \cite{
 Gopalakrishnan09aconvergent,cockburn2014multigrid,kronbichler2018performance,chen2015auxiliary} may also work for the case of multinumerics. However,
 to the best of our knowledge we are not able to find such a study yet and it is beyond the current scope of this paper. In our future work we intend to compare these different 
 algorithms along with our approach in the context of multinumerics.

We consider the following permeability field 
 \[ {\bf K} = (1+1/2\sin(4\pi x)\cos(3\pi y)){\bf I},\]
 and  $f=1$. We consider unit square and use either RT1
mixed finite elements or first order NIPG discontinuous Galerkin as shown in Figure \ref{tab:multinumerics} (left).
We plot the finest mesh in the left of Figure~\ref{tab:multinumerics} and a coarsening strategy with 4 levels in Figure~\ref{fig:dg_mixed_levels}. We point out that (many or all)  macro-elements on the coarser grids contain both RT1 and NIPG fine grid elements.
\begin{figure}[h!b!t!]
\begin{minipage}{0.25\linewidth}
\begin{center}
\includegraphics[scale = 0.22, trim=1cm 1cm 3cm 0cm,clip=true]{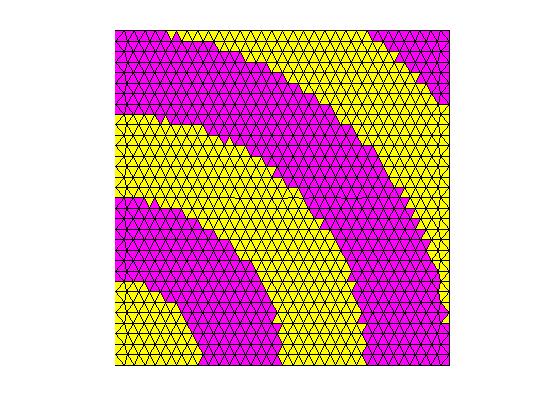}
\end{center}
    \end{minipage}
    \begin{minipage}{0.75\linewidth}
 \begin{center}
 \begin{tabular}{| r || c | c | c | c | c |} \hline
     Levels & 3 & 4 & 5 & 6 & 7 \\
     \hline
     Elements & 120 & 496 & 2016 & 8128 & 32640 \\
     \hline
     Iterations & 9 & 10 & 9 & 9 & 9 \\
     \hline
  \end{tabular}
 \end{center}
    \end{minipage}
    \caption{Example VI: On the left, NIPG-H elements are purple and RT1 elements  are yellow. On the right are the number of iterations for the multigrid solver using multinumerics.}
\label{tab:multinumerics}
\end{figure}
\vspace{-4mm}
\begin{figure}[h!b!t!]
\begin{center}
\includegraphics[scale = 0.22]{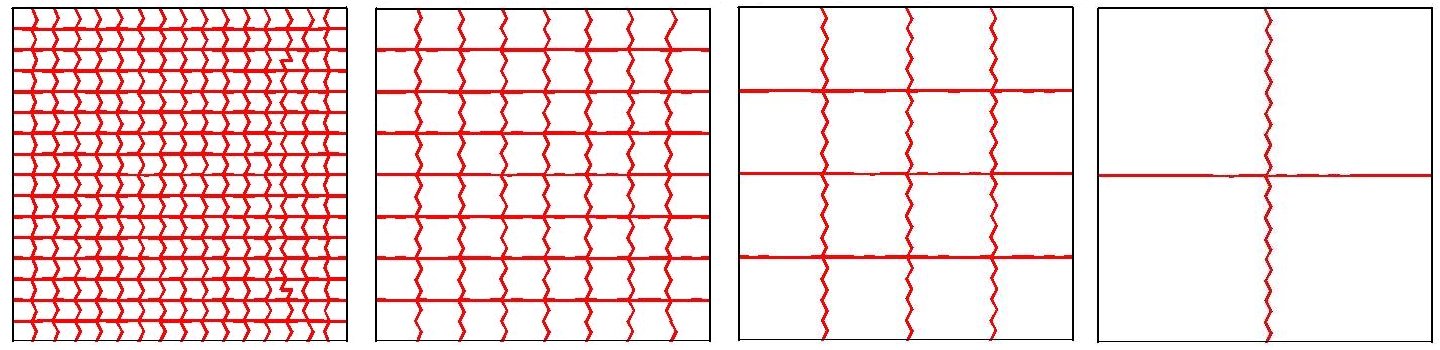}
\end{center}
\caption{Example VI: Coarsening strategy.}
\label{fig:dg_mixed_levels}
\end{figure}

On the right of Figure \ref{tab:multinumerics} are the iteration
counts using the multigrid solver. It can be observed that the
performance, i.e. the number of iterations, using multinumerics is
nearly identical to those for NIPG-H in Section \secref{exp1} despite
the fact that different numerical methods are used throughout the
domain and the macro-elements are neither triangles nor
quadrilaterals. This flexibility is one of the highlighted
characteristics of our multigrid algorithm.


\section{Conclusion}
\seclab{conclusion}
We have proposed a unified DtN-based geometric multigrid algorithm for hybridized high-order finite
element methods. Our approach differs significantly from the previous
attempts in the sense that the intergrid transfer operators are
physics-based energy-preserving and the coarse grid operators are
discretized DtN maps on every level. These operators completely avoid
upscaling of parameters, such as permeability and source, to coarse
scales. As our numerical results have indicated, they also allow for
multinumerics and variable coefficients without deteriorating the
performance of the multigrid algorithm. We have presented several
numerical examples with hybridized discontinuous Galerkin (HDG)
methods, hybridized interior penalty discontinuous Galerkin (IPDG)
methods and reported several  observations. For HDG methods,
with stabilization $\tau=1$ or $1/h_{min}$ (depending on the
example), we obtain almost perfect scalability in mesh size and
solution order using multigrid preconditioned GMRES. For all the other
hybridized IPDG methods, with a stabilization of the form
$\tau=\mc{O}\LRp{\p^2/h_{min}}$ similary scalability is
observed. However, the multigrid algorithm applied to IPDG seems to be less robust compared to HDG with
respect to stabilization parameters and coarsening strategies
especially for highly unstructured meshes. We have also reported scalable
results for the proposed multgrid method when applying to  multinumeric discretizations with mixed methods in certain
parts of the domain and DG methods in other parts.
We are currently  extending  our multigrid algorithm to convection-diffusion, Stokes, Oseen and incompressible resistive magnetohydrodynamic (MHD) equations. Ongoing research is to provide a rigorous analysis of the proposed multigrid algorithm and we will report our findings elsewhere in the near future.

\section*{Acknowledgements}
T.~Wildey's work is supported by the Office of Science Early Career Research Program. S. Muralikrishnan and T. Bui-Thanh are partially
supported by the DOE grant DE-SC0018147 and NSF grant NSF-DMS1620352. We are grateful for the support. The authors would like to thank the anonymous referees for their critical and useful comments that significantly improved the paper.  

\bibliography{references,ceo,DtNbib}

\end{document}